\documentclass[a4paper,11pt,leqno]{article}

\usepackage[latin1]{inputenc}
\usepackage[T1]{fontenc} 
\usepackage[english]{babel}
\usepackage{verbatim}
\usepackage{calligra}
\usepackage{enumerate}
\usepackage{dsfont}
\usepackage{amsfonts}
\usepackage{amsmath}
\usepackage{amsthm}
\usepackage{amssymb}
\usepackage{mathtools}
\usepackage{mathrsfs}
\usepackage{color}
\usepackage{graphicx}
\usepackage{bm}
\usepackage{srcltx}

\usepackage{tikz}
\usepackage[retainorgcmds]{IEEEtrantools}

\usepackage{graphicx}		
\usepackage[hypcap=false]{caption}

\theoremstyle{definition}
\newtheorem{defin}{Definition}[section]

\theoremstyle{plane}
\newtheorem{thm}[defin]{Theorem}
\newtheorem{prop}[defin]{Proposition}
\newtheorem{cor}[defin]{Corollary}
\newtheorem{lemma}[defin]{Lemma}

\newcommand{\tbf}{\textbf}

\newcommand{\tsl}{\textsl}

\newcommand{\mc}{\mathcal}
\newcommand{\mf}{\mathfrak}

\newcommand{\wtilde}{\widetilde}
\newcommand{\vphi}{\varphi}

\renewcommand{\k}{\kappa}

\renewcommand{\t}{\tau}

\newcommand{\de}{\delta}
\renewcommand{\o}{\omega}

\newcommand{\lan}{\langle}
\newcommand{\ran}{\rangle}

\newcommand{\R}{\mathbb{R}}

\newcommand{\N}{\mathbb{N}}

\newcommand{\T}{\mathbb{T}}

\renewcommand{\div}{{\rm div}\,}
\newcommand{\divh}{{\rm div}_{\!h}}

\newcommand{\curl}{{\rm curl}\,}

\newcommand{\dx}{ \, {\rm d} x}
\newcommand{\dt}{ \, {\rm d} t}

\newcommand{\loc}{{\rm loc}}

\allowdisplaybreaks

\def\d{\partial}
\def\div{{\rm div}\,}
\newcommand{\dd}{{\rm d}}

\begin{document}

\title{\textsc{\Large{\textbf{Global Yudovich-type solutions to a reduced model for micropolar fluids with zero viscosity}}}}

\author{\normalsize \textsl{Francesco Fanelli}$\,^{1,2,3} \qquad$ and $\qquad$
\textsl{Pedro Gabriel Fern\'andez-Dalgo}$\,^{1}$
\vspace{.5cm} \\
\footnotesize{$\,^{1}\;$ \textsc{BCAM -- Basque Center for Applied Mathematics}} \\ 
{\footnotesize Alameda de Mazarredo 14, E-48009 Bilbao, Basque Country, SPAIN} \vspace{.2cm} \\
\footnotesize{$\,^{2}\;$ \textsc{Ikerbasque -- Basque Foundation for Science}} \\  
{\footnotesize Plaza Euskadi 5, E-48009 Bilbao, Basque Country, SPAIN} \vspace{.2cm} \\
\footnotesize{$\,^{3}\;$ \textsc{Universit\'e Claude Bernard Lyon 1}, {\it Institut Camille Jordan -- UMR 5208}} \\ 
{\footnotesize 43 blvd. du 11 novembre 1918, F-69622 Villeurbanne cedex, FRANCE} \vspace{.3cm} \\
%
%
\footnotesize{Email addresses: $\;$ \ttfamily{ffanelli@bcamath.org}}, $\;$
\footnotesize{\ttfamily{pfernandez@bcamath.org}}
\vspace{.2cm}
}

\date\today

\maketitle

\subsubsection*{Abstract}
{\footnotesize 

In this paper, we study the well-posedeness at low regularity of a two-dimensional system obtained as a reduced model for micropolar fluid dynamics.
At the mathematical level, the system presents a coupling between an Euler-type equation for the two-dimensional velocity field of the fluid and an
advection-diffusion equation for the scalar microrotation field.

For this model, we prove global existence and uniqueness of Yudovich-type solutions, namely weak solutions for which the vorticity is only bounded (with some
additional integrability property) and the microrotation field remains bounded and of finite energy.
To the best of our knowledge, this is the first result which extends the genuine Yudovich framework to a system obtained by perturbing the incompressible Euler equations
with some sort of heterogeneity.

}

\paragraph*{\small 2020 Mathematics Subject Classification:}{\footnotesize 35Q35 
(primary);
35R05, 
76B47, 
35D30 
(secondary).}

\paragraph*{\small Keywords: }{\footnotesize micropolar fluids; zero viscosity; Yudovich solutions; global existence; uniqueness.
}

%
%

\section{Introduction} \label{s:intro}

In this paper, we are interested in the well-posedness theory of a reduced, two-dimensional model
describing the dynamics of micropolar fluids (see Eringen \cite{Erin}).
Denote by $u=(u^1,u^2)\in\R^2$ the velocity field of the fluid, $\Pi\in\R$ its pressure field and by $m\in\R$ a scalar quantity representing
its microrotation field; we also define $\o$ to be the vorticity of the fluid, given by the formula
\[
 \o\,:=\,\curl(u)\,=\,\d_1u^2\,-\,\d_2u^1\,. 
\]

The equations we will be considering throughout this work read as follows:
\begin{equation} \label{eq:2D-microp}
\left\{\begin{array}{l}
        \d_tu\,+\,(u\cdot\nabla)u\,+\,\nabla\Pi\,=\,-\,\alpha\,\nabla^\perp m \\[1ex]
        \d_tm\,+\,u\cdot\nabla m\,-\,\k\,\Delta m\,=\,\alpha\,\o \\[1ex]
        \div u\,=\,0\,.
       \end{array}
\right.
\end{equation}
where the constants $\alpha>0$ and $\k>0$ are both assumed positive and where we have used the notation
$\nabla^\perp m\,:=\,(-\d_2m , \d_1m)$.

We set system \eqref{eq:2D-microp} on
\[
 (t,x)\,\in\,\R_+\times\Omega\,,\qquad\qquad \mbox{ with }\qquad \Omega\,=\,\R^2\ \mbox{ or }\ \T^2\,,
\]
and we supplement it with the initial condition
\[
 \big(u,m\big)_{|t=0}\,=\,\big(u_0,m_0\big)\,.
\]
Suitable assumptions on the two-dimensional vector field $u_0$ and on the scalar field $m_0$ will be specified later on.

The goal of the present paper is to set down a weak solutions theory \tsl{\`a la Yudovich} for equations \eqref{eq:2D-microp}. Before presenting our main result,
which will be formulated in Subsection \ref{ss:results} below, let us spend a few words on the derivation of equations \eqref{eq:2D-microp}.

\subsection{Micropolar fluids} \label{ss:derivation}
The micropolar fluid equations were introduced by Eringen \cite{Erin} in the late 1960s as a continuum model for fluids with internal microstructure,
allowing each material point to undergo independent rigid micro-rotations in addition to the usual translational motion.
This theory modifies the classical Navier-Stokes equations by incorporating an additional vector field describing the microrotation of fluid particles and by introducing
new constitutive parameters accounting for micro-inertia and rotational viscosity effects.

The dynamics of a general, three-dimensional, homogeneous micropolar fluid can be described through evolution equations on its velocity field $U\in\R^3$,
its pressure field $\pi\in\R$ and its microrotation field $M\in\R^3$.
The equations of motion take the following form (see \tsl{e.g.} Part I, Chapter 1, Section 3 of \cite{Luk} for details):
\begin{equation} \label{eq:microp}
\left\{\begin{array}{l}
        \d_tU\,+\,(U\cdot\nabla)U\,+\,\nabla\pi\,-\,(\mu+\mu_r)\Delta U\,=\,2\,\mu_r\,\curl M \\[1ex]
        \d_tM\,+\,(U\cdot\nabla)M\,-\,(c_a+c_d)\,\Delta M\,-\,(c_0+c_d-c_a)\,\nabla\div M\,+\,4\,\mu_r\,M\,=\,2\,\mu_r\,\curl U \\[1ex]
        \div U\,=\,0\,.
       \end{array}
\right.
\end{equation}
Despite we use the same notation as in system \eqref{eq:2D-microp}, notice that, in system \eqref{eq:microp}, all the space-differentiation operators are three-dimensional.
The coefficients appearing therein have the following meaning:
\begin{itemize}
 \item $\mu\geq 0$ represents the dynamic viscosity coefficient of the fluid;
 \item $\mu_r\geq0$ denotes the dynamic microrotation viscosity;
 \item $c_0$, $c_d$ and $c_a$ are the so-called coefficients of angular viscosities.
\end{itemize}

From a modeling viewpoint, the micropolar system provides a refined description of \emph{complex fluids with microstructure}, including suspensions,
liquid crystals, and fluids with rigid or semi-rigid particles. The classical velocity field $u$ represents the transport of the continuum, namely the
coarse-level advection observable at the fluid scale, while the intrinsic microrotation $M$ encodes the internal spinning of micro-elements,
which remains invisible at the classical level.
For a detailed discussion about the physical deduction of this model and a physical interpretation of the constants, we refer to \cite{Luk}.

After the mathematical formulation of Eringen \cite{Erin}, foundational existence and uniqueness results for weak and strong solutions in various viscous settings were soon developed.
Classical existence theorems of finite energy weak solutions, analogous to those available for the Navier-Stokes equations, were established by Galdi and Rionero \cite{GaldiRionero1977}. Complementary approaches using functional frameworks such as critical Besov spaces were later pursued by Chen \cite{Chen2012}, who proved global well-posedness of the micropolar system in scaling-critical norms. Other lines of research have investigated regularity and uniqueness properties under specific structural assumptions on viscosities and boundary conditions \cite{YangWang2019}.
These works highlight the rich variety of analytical methods applicable to micropolar models and the nontrivial coupling between the velocity and
microrotation fields.

Despite the mathematical formulation of the micropolar fluid equations \eqref{eq:microp} looks as an extension of the incompressible Navier-Stokes equations,
some words of caution are in order. 
%
As a matter of fact, setting $M\equiv 0$ in \eqref{eq:microp} does not allow to recover the Navier-Stokes system, but instead enforces the condition
$\curl U=0$, which, together with the incompressibility constraints $\div U=0$, yields the triviality of any ``reasonable'' solution $U\equiv0$.
The micropolar system should therefore \emph{not} be regarded as a generalization of the Navier--Stokes equations, inasmuch as it encodes a fundamentally different structure.

\subsection{Derivation of the reduced model} \label{ss:reduced-deriv}

System \eqref{eq:2D-microp} is obtained from the general micropolar fluid equations \eqref{eq:microp} by performing the following reductions:
\begin{enumerate}[(i)]
 \item first of all, after decomposing vectors $A\in\R^3$ as $A=(A_h,A^3)$, with $A_h=(A^1,A^2)\in\R^2$ and $A^3\in\R$,
we restrict our attention to special solutions of the form
\[
 U\,=\,\big(u_h(t,x_h) , 0\big)\qquad \mbox{ and }\qquad M\,=\,\big(0,0,m(t,x_h)\big)\,,\qquad\qquad \mbox{ with }\qquad \divh u_h=0\,,
\]
where we have denoted $\divh$ to be the two-dimensional divergence of a 2-D vector field, that is $\divh u_h\,=\,\d_1u^1\,+\,\d_2u^2$;
observe in particular that, with this choice of $U$ and $M$, one has $\div M=0$, $\curl M = ( \d_2m, -\d_1 m, 0)$ and $\curl U = (0,0, \o)$;
\item we define $\k := c_d + c_a$ and $\alpha = 2\,\mu_r$, and we take $u = u_h$;
\item we arbitrarily take $\mu + \mu_r =0$ in the viscosity term $\Delta U$ appearing in \eqref{eq:microp};
\item in an equally arbitrarily way, we erase the damping term $+\,4\,\mu_r\,M$ appearing on the left-hand side of the equation for $M$.
\end{enumerate}
The last items (iii) and (iv) are in fact unphysical, as the former obviously violates the principles $\mu\geq0$ and $\mu_r\geq0$ stated above, whereas the latter
is not consistent with the fact that we keep the term $2\,\mu_r\,\curl U$ on the right-hand side of the second equation.
Nonetheless, our motivation here is to look for some sort of stability of previous well-posedness results for \eqref{eq:microp} with respect to viscosity.
It should be noted that the answer to this question does not appear, at least at a first sight, completely obvious.
Indeed, in absence of viscous dissipation
and damping, the energy balance for equations \eqref{eq:microp} highlighted in \cite{GaldiRionero1977} breaks down.
More precisely, when performing an energy estimate (that is, multiplying in $L^2$ the first equation by $U$ and and the second equation by $M$),
the contribution of the right-hand sides of the first and second equations in \eqref{eq:microp} does not sum up to $0$; on the contrary,
they give rise to two terms of the same sign. Viscosity and damping together have (to some extent) a stabilising effect, as their contribution serves precisely
to counterbalance the energy production coming from the terms $\curl U$ and $\curl M$.

The reduced model \eqref{eq:2D-microp} proposed in this article suffers of the same pathological structure: since we neglect both
viscosity and damping terms there, the energy has an exponential growth, which may then be a source of instability.
Of course, system \eqref{eq:2D-microp} already posseses a dissipation mechanism on $m$, through the microrotation
viscosity term $-\k\Delta m$: as far as regularity is concerned, leveraging this effect is enough to compensate for the loss of one derivative due to the presence
of the term $-\alpha\nabla^\perp m$ appearing in the first equation of \eqref{eq:2D-microp}. Yet, this may \tsl{a priori} not be enough to propagate
regularity for long times and recover global well-posedness of the system.

To conclude this part, we point out that a different two-dimensional reduced model for micropolar fluids was recently studied in \cite{BL-C-S}. In that model,
viscosity was considered, 
whereas the equation of the variable $m$ (called $b$ in that paper) presented no diffusion, but only damping.
The gain of regularity on the velocity field was used to compensate the loss of derivatives introduced by the presence of the forcing term
$-\alpha\nabla^\perp m$ on the right-hand side of the equation for $u$, and damping together with heat kernel estimates enabled the authors to get time decay.
In our paper, instead, the regularity mechanism works the other way around (we gain with respect to $m$, to compensate the loss coming from $\alpha \o$ in its equation);
in addition, our analysis differs from \cite{BL-C-S} both concerning the proof of existence (we use here a different regularisation procedure) and uniqueness (our
approach is mainly inspired by the original argument of Yudovich \cite{Yud_1963}).
It is fair to mention that, however, the well-posedness result of \cite{BL-C-S} holds true for less regular initial data then ours.

%
%
%

\subsection{Statement of the main result} \label{ss:results}

In the present paper, we address the question of well-posedness of the reduced micropolar fluid system \eqref{eq:2D-microp}. Our goal is to construct
\emph{global solutions} in a weak regularity framework, which mimics the \emph{Yudovich theory} \cite{Yud_1963} for the classical incompressible Euler equations
(see \tsl{e.g.} the monographs \cite{Maj-Bert} and \cite{BCD} for more details).

Before presenting the precise statement, a few definitions are in order. First of all, given a two-dimensional vector field $v$ over $\R^2$, we define
the scalar field
\[
 \curl(v)\,:=\,\d_1v^2\,-\,\d_2v^1
\]
to be the two-dimensional ``vorticity'' associated to $v$.
Next, we recall the definition of log-Lipschitz functions over $\R^d$ (here $d=2$, the only case of interest for us).

\begin{defin} \label{def:LL}
A function $f\in L^\infty(\R^d)$ is said to be \emph{log-Lipschitz} continuous, and we write $f\in LL(\R^d)$, if the quantity
\[
|f|_{LL}\,:=\,\sup_{z,y\in\R^d,\,0<|y|<1}
\left(\frac{\left|f(z+y)\,-\,f(z)\right|}{|y|\,\log\left(1\,+\,\frac{1}{|y|}\right)}\right)\,<\,+\infty\,.
\]
We define the log-Lipschitz norm as $\|f\|_{LL}\,:=\,\|f\|_{L^\infty}\,+\,|f|_{LL}$.
\end{defin}

Recall that vector fields which are $L^1_T(LL)$ admit a uniquely defined flow, see \tsl{e.g.} Chapter 3 of \cite{BCD} for details.

After these premises, we are ready to present the main result of this paper, which is contained in the following statement.

\begin{thm} \label{th:global-wp}
Let $p_0\,\in\,\,]1,2[\,$ fixed. Take an initial datum $\big(u_0,m_0\big)$, with $\div u_0=0$, such that, after defining the initial vorticity
$\o_0\,:=\,\curl(u_0)\,=\,\d_1u_0^2-\d_2u_0^1$, one has 
\[
 \o_0\,,\,m_0\;\in\,L^{p_0}(\R^2)\cap L^\infty(\R^2)\,.
\]

 Then there exists a unique global in time solution $\big(u,m\big)$ to system \eqref{eq:2D-microp}, related to the initial datum $\big(u_0,m_0\big)$ and such
that
\[
 \o\,,\,m\;\in\,L^\infty_\loc\big(\R_+;L^{p_0}(\R^2)\cap L^\infty(\R^2)\big)\qquad \mbox{ and }\qquad 
 \nabla m\,\in\,L^2_\loc\big(\R_+;L^{2}(\R^2)\big)\,,
\]
where we have defined $\o\,:=\,\curl(u)\,=\,\d_1u^2-\d_2u^1$ the vorticity of the fluid. In addition, after setting $q_0\,\in\,\,]2,+\infty[\,$
such that $1/q_0\,=\,1/p_0\,-\,1/2$, one has 
\[
 u\,\in\,L^\infty_\loc\big(\R_+;L^{q_0}(\R^2)\cap LL(\R^2)\big)\,.
\]
\end{thm}

To the best of our knowledge, this is the first result proving (global) existence and uniqueness of Yudovich-type solutions for a sort of non-homogeneous perturbation of
the incompressible Euler equations. As a matter of fact, previous results for \tsl{e.g.} the ideal MHD equations \cite{Hmidi}, the inviscid and non-diffusive Boussinesq
equations \cite{Hass-Hm} and the non-homogeneous incompressible Euler equations \cite{F_2012} prove (local) well-posedness in a subclass
of Yudovich data, more precisely the one obtained by imposing conormal regularity assumptions, in spirit of the pioneering work by J.-Y. Chemin
about regularity of vortex patches \cite{Ch_1991, Ch_1993}.

It must be noted that global well-posedness results in a low regularity framework exist also in the context of the inviscid Boussinesq system in presence of
temperature diffusion, see \tsl{e.g.} \cite{Hmidi-Zer, Paicu-Zhu, Hm-Hou-Z}. However, we remark that the regularity required on the initial data and the one propagated
for the corresponding solution always places at a slightly higher level than the one prescribed by the analogue of Yudovich theory. In addition,
we point out that the system considered in those papers, even though possessing a very similar structure as ours \eqref{eq:2D-microp},
is less singular, inasmuch as the right-hand side of the momentum equation (that is, the equation for the velocity field $u$) present no derivatives acting on
$m$ and, at the same time, there is no forcing term appearing on the right-hand side of the equation for $m$.

Finally, let us mention that, in \cite{H-K-R}, Hmidi, Keraani and Rousset proved global well-posedness for an inviscid Boussinesq system with critical diffusion.
At the level of the structure of the equations and of the order of the operators involved, their system is very much related to the equations considered here.
In fact, we will borrow from that paper a crucial idea in order to propagate the $L^\infty$ norm of the vorticity (see more details in Subsection \ref{ss:leb} below).
However, owing to the non-locality of the operators involved in their computations, the results of \cite{H-K-R} are formulated in a (low regularity, yet) finer setting
than the pure Yudovich class. Of course, here considering the genuine Yudovich framework is made possible, thanks to the locality (and linearity) of the operators involved on
the right-hand side of our equations. At the same time, this makes less clear how to get global well-posedness in spaces constructed over $L^\infty$: this question
will be addressed in the forthcoming paper \cite{F-FD-MM}.

%

\subsection*{Organisation of the paper}

After this introduction, we now give an overview of the rest of the paper.

In Section \ref{s:a-priori}, we show \tsl{a priori} bounds for smooth solutions of the reduced micropolar fluid system \eqref{eq:2D-microp}: we recall
some basic facts about the Biot-Savart law and, after that, we exhibit global in time bounds for the Lebesgue norms of the vorticity $\o$ and the microrotation field
$m$.

In Section \ref{s:existence}, we give the rigorous proof of existence of Yudovich-type solutions to system \eqref{eq:2D-microp}. The argument relies on solving
the original equations for smoothed initial data and on applying a global well-posedness result of strong solutions from \cite{F-FD-MM}.
Convergence of the sequence of regular solutions towards a true Yudovich solution is performed by a compactness argument.

Finally, in Section \ref{s:uniqueness} we prove uniqueness of the constructed solution in the Yudovich class. The proof is inspired by the original work
of Yudovich \cite{Yud_1963}, which can be carried out also in our setting with minor modifications. However, an important point in this argument consists
in computing an energy estimate for the difference of velocity fields: justifying that this is possible in our framework (remark that $u$ does not belong
to $L^2$, in general) is a crucial and delicate point of our work.

\subsection*{Notation} 

Let us fix some notation which will be freely used throughout this paper.

Given a Banach space $\mf B$ over $\R^2$, we will adopt the same notation $\mf B(\R^2)$ (or simply $\mf B$) for scalar, vector-valued and matrix valued functions.
Typically, we will resort to the longer notation $\mf B(\R^2)$ when formulating the assumptions and the statements, and in
important formulas.

For an interval $I\subset \R$ and $\mf B$ as above,
we denote by $\mc C\big(I;\mf B\big)$ the space of continuous bounded functions on $I$ with values in $\mf B$. For any $p\in[1,+\infty]$,
the symbol $L^p\big(I;\mf B\big)$ stands for the space of measurable functions on $I$ such that the map $t\mapsto \left\|f(t)\right\|_{\mf B}$ belongs to $L^p(I)$.
When $I=[0,T]$, we will often use the shorten notation $L^r_T(\mf B)\,=\,L^r\big([0,T];\mf B\big)$, whereas we will resort to the full notation
in statements and centered formulas.
We will denote the space of test-functions over some set $Q$ equivalently by $\mc C^\infty_0(Q)$ or by $\mc D(Q)$, where $Q$ for us
will typically be either $\R^2$ or $[0,T]\times\R^2$. In the latter case, it is understood that the test function vanishes at time $t=T$.

In our estimates we will often avoid to write the explicit multiplicative constants which allow to pass from one line to the other. Thus, we will write
$A\,\lesssim\, B$ meaning that there exists a universal constant $C>0$, not depending on the solutions nor on the data (in the latter case, this will be pointed out),
such that $A\,\leq\,C\,B$.



Given a two-dimensional vector field $v$, we define $v^\perp$ to be its rotation of angle $\pi/2$. More precisely, if $v\,=\,\big(v^1,v^2\big)$, then
$v^\perp\,=\,\big(-v^2,v^1\big)$. Similarly, we define the operator $\nabla^\perp$ as $\nabla^\perp\,=\,\big(-\d_2,\d_1\big)$.
Finally, given a two-dimensional vector field $v\,=\,\big(v^1,v^2\big)$, we define its $\curl$ as
$\curl(v)\,=\,\d_1v^2\,-\,\d_2v^1$.

\section*{Acknowledgements}

{\small

This work has been partially supported 
by the Basque Government through the BERC 2022-2025 program and by the Spanish State Research Agency through the BCAM Severo Ochoa excellence accreditation
CEX2021-001142.

The first author also aknowledges the support of the French National Research Agency (ANR) through the project CRISIS (ANR-20-CE40-0020-01).
This work has been co-funded by the European Union through the project HORIZON-MSCA-2022-COFUND-01-SmartBRAIN3-101126600.

}

\section{\tsl{A priori} estimates} \label{s:a-priori}

In this section, we derive \tsl{a priori} estimates for (supposed to exist) smooth solutions to equations \eqref{eq:2D-microp}. Actually, in order to define
solutions in a Yudovich framework, it is convenient to reformulate the system in terms of the vorticity of the fluid, rather than on the velocity field $u$
itself: this is the goal of Subsection \ref{ss:vort}, devoted to the statement of important static estimates which come from the Biot-Savart law, and
Subsection \ref{ss:vort-form}, where we prove the equivalence between the vorticity formulation and the original formulation of the equations.
Then, in Subsection \ref{ss:leb}, we introduce the auxiliary variable $\Gamma$ and prove global $L^p$ bounds for the couple $\big(\Gamma,m\big)$, which in turn
yield bounds also for the vorticity $\o$.

\subsection{The vorticity and the Biot-Savart law} \label{ss:vort}

Let $\big(u,m\big)$ be a regular solution to the reduced micropolar fluid equations \eqref{eq:2D-microp}, and let
\[
 \o\,:=\,\curl(u)\,=\,\d_1u^2\,-\,\d_2u^1
\]
be the vorticity of the fluid. Since $u$ is divergence-free, it can be recast in terms of $\o$ by solving the \emph{Biot-Savart law},
which in two space dimensions takes the form
\begin{equation} \label{eq:BS}
u\,=\,-\,\nabla^\perp(-\Delta)^{-1}\o\,=\,\frac{1}{2\pi}\,\int_{\R^2}\dfrac{1}{|x-y|^2}\,(x-y)^\perp\,\o(y)\,\dd y\,. 
\end{equation}

From the Biot-Savart law, one can deduce the following classical estimates. For the sake of completeness,
we sketch here their proof; we refer to \tsl{e.g.} Chapter 8 of \cite{Maj-Bert} for more details.

\begin{prop} \label{p:BS}
Let $p_0\,\in\,\,]1,2[\,$ and define $q_0\,\in\,\,]2,+\infty[$ as
\[
\frac1{q_0}\,:=\,\frac1{p_0}\,-\,\frac12\,.
\]
Assume that the vorticity $\omega$ satisfies
$\omega\in L^{p_0}(\mathbb R^2)\cap L^\infty(\mathbb R^2)$,
and let $u$ be the associated velocity field given by the Biot-Savart law \eqref{eq:BS}.

Then, the following properties hold true:
\begin{enumerate}
\item[\rm(i)]
$u\in L^{q_0}(\R^2)\cap L^\infty(\R^2)$; moreover, there exists a universal constant $C>0$ such that
\[
\|u\|_{L^{q_0}}\,\leq\,C\,\|\omega\|_{L^{p_0}} \qquad \mbox{ and }\qquad
\|u\|_{L^\infty}\,\leq\,C\,\big(\|\omega\|_{L^{p_0}}\,+\,\|\o\|_{L^\infty}\big)\,;
\]
\item[\rm(ii)]
the velocity field $u$ is log-Lipschitz continuous, in the sense of Definition \ref{def:LL}, and there exists
$C>0$ such that
\[
|u|_{LL}\,\leq\,C\,\big(\|\omega\|_{L^{p_0}}\,+\,\|\o\|_{L^\infty}\big)\,. 
\]
\end{enumerate}
\end{prop}

\begin{proof}
 The fact that $u\in L^{q_0}(\R^2)$, together with the claimed bound on the corresponding norm, is an immediate consequence of the Hardy-Littlewood-Sobolev inequality
(see \tsl{e.g.} Theorem 1.7 in \cite{BCD}), after noticing that the kernel $K$ of the Biot-Savart law, namely
 \[
  K(z)\,:=\,\frac{1}{2\pi}\,\frac{1}{|z|^2}\,z^\perp\,,
 \]
is a homogeneous function of degree $-1$.

Let us bound $u$ in $L^\infty$. For this, one can use the decomposition
\[
u(x)\,=\,\int_{B(0,1)} K(y)\,\omega(x-y)\,\dd y\,+\,\int_{\R^2\setminus B(0,1)}K(y)\,\omega(x-y)\,\dd y\,,
\]
where $B(0,1)$ denotes the ball in $\R^2$ of center $0$ and radius $1$. Thus, the claimed bound follows from Young inequality for convolutions,
after using that $K$ is locally integrable near the origin, while at infinity it belongs to any $L^p$ space, for $p>2$.
This completes the proof of item (i).

For proving item (ii) of the statement, we fix $x\in \R^2$ and $y\in B(x,1)$; we set
\[
r\,:=\,|x-y|\,\leq\,1\,.
\]
Then, we write
\[
u(x)-u(y)\,=\,\int_{\R^2}\big(K(x-z)-K(y-z)\big)\,\omega(z)\,\dd z\,,
\]
from which we can estimate
\[
\big|u(x)-u(y)\big|\,\leq\,
\int_{\R^2}|\omega(z)|\, \big|K(x-z)-K(y-z)\big|\,\dd z\,.
\]
We split the integral into a near-field and a far-field contribution:
\begin{align}
\label{est:near+far}
    \int_{\R^2}|\omega(z)|\, \big|K(x-z)-K(y-z)\big|\,\dd z
    &=
    \int_{|x-z|\leq 2r} \!\!\! |\omega(z)|\,\big|K(x-z)-K(y-z)\big|\,\dd z \\
\nonumber
    &\qquad +\,\int_{|x-z|>2r} \!\!\! |\omega(z)|\,\big|K(x-z)-K(y-z)\big|\,\dd z\,. 
\end{align}
We start by considering the near-field contribution
\begin{align*}
 I_{\rm near}\,&:=\,\int_{|x-z|\leq 2r} \!\!\! |\omega(z)|\,\big|K(x-z)-K(y-z)\big|\,\dd z \\
 &\leq\,\left\|\o\right\|_{L^\infty}\,\int_{|x-z|\leq 2r} \!\!\!\big|K(x-z)-K(y-z)\big|\,\dd z\,.
\end{align*}
Using the triangle inequality, we obtain
\[
\big|K(x-z)-K(y-z)\big|\,\leq\,\big|K(x-z)\big|\,+\,\big|K(y-z)\big|\,
\leq\,
C\,\left(\frac1{|x-z|}\,+\,\frac1{|y-z|}\right)\,.
\]
Therefore, we deduce that
\[
I_{\mathrm{near}}\,\lesssim\,\left\|\o\right\|_{L^\infty}\,\left(\int_{|x-z|\le 2r} \frac{\dd z}{|x-z|}\,+\,
\int_{|x-z|\le 2r} \frac{\dd z}{|y-z|}\right)\,,
\]
where a change of variables immediately gives
\[
\int_{|x-z|\le 2r} \frac{\dd z}{|x-z|}
\,=\, \int_{|w|\le 2r} \frac{\dd w}{|w|}\,\lesssim\, r\,.
\]
On the other hand, we note that $|x-z|\,\leq\, 2\,r$ implies
\[
|y-z|\,\leq\, |y-x|\,+\,|x-z|\,\leq\, 3\,r\,,
\]
whence $\big\{|x-z|\leq 2r\big\}\,\subset\,\big\{|y-z|\le 3r\big\}$. This implies that
\[
\int_{|x-z|\le 2r} \frac{\dd z}{|y-z|}\,\leq\,\int_{|y-z|\leq 3r} \frac{\dd z}{|y-z|}\,=\,\int_{|w|\leq 3r} \frac{\dd w}{|w|}
\,\lesssim\, r\,.
\]
Combining the two bounds finally yields
\begin{equation} \label{est:near}
I_{\mathrm{near}}\,\lesssim\,\left\|\o\right\|_{L^\infty}\, r\,.
\end{equation}
Next, we focus on the far-field term
\[
 I_{\rm far}\,:=\,\int_{|x-z|> 2r} \!\!\!  |\omega(z)|\, \big|K(x-z)-K(y-z)\big|\,\dd z\,.
\]
Since $|x-z|>2r$, we can use the smoothness of the kernel and, by the mean value theorem, write
\[
\big|K(x-z)-K(y-z)\big|\,\leq\,
|x-y|\, \sup_{\xi\in[x,y]} \big|\nabla K(\xi-z)\big|\,,
\]
where $[x,y]$ denotes the line segment in $\R^2$ joining the points $x$ and $y$.
Observe that we have
$\big|\nabla K(w)\big|\,\lesssim\,|w|^{-2}$.
At the same time, by triangular inequality, we can write
\[
 |\xi-z|\,\geq\,|x-z|\,-\,|\xi-x|\,\geq\,|x-z|\,-\,r\,\geq\,\frac{1}{2}\,|x-z|\,.
\]
Therefore, gathering the previous inequalities, we obtain
\[
\big|K(x-z)-K(y-z)\big|\,\lesssim\,r\,|x-z|^{-2}\,,
\]
from which we deduce the following bound:
\begin{align*}
I_{\mathrm{far}}\,&\lesssim\,r\,\int_{|x-z|>2r} |\omega(z)|\, \frac{\dd z}{|x-z|^2} \\
&\lesssim\,
r\,\left\|\o\right\|_{L^\infty}\,\int_{2r}^{1} \frac{\dd\rho}{\rho}\,+\,
r\,\int_{|x-z|>1} |\omega(z)|\, \frac{\dd z}{|x-z|^2}\,. 
\end{align*}
From this relation, thanks to the fact that $p_0\in\,]1,2[\,$ and that $|z|^{-2q}$ is integrable over $\R^2$ for any $q>1$,
we immediately get the following estimate for $I_{\rm far}$:
\begin{align}
\label{est:far}
I_{\mathrm{far}}\,&\lesssim\,r\,|\log r|\,\left\|\o\right\|_{L^\infty}\,+\,r\,\left\|\o\right\|_{L^{p_0}}\,.
\end{align}
At this point, inserting inequalities \eqref{est:near} and \eqref{est:far} into \eqref{est:near+far} immediately yields that
$u$ belongs to the log-Lipschitz space $LL(\R^2)$, together with the claimed bound on its log-Lipschitz seminorm.

This concludes the proof of the proposition.
\end{proof}

\subsection{The vorticity formulation of the equations} \label{ss:vort-form}
With Proposition \ref{p:BS} at hand, we can reformulate system \eqref{eq:2D-microp} in an equivalent way as a system on the new unknowns $(\o,m)$. Since in two space
dimensions the stretching term is identically zero, we get
\begin{equation} \label{eq:o-m}
\left\{\begin{array}{l}
        \d_t\o\,+\,u\cdot\nabla\o\,=\,-\,\alpha\,\Delta m \\[1ex]
        \d_tm\,+\,u\cdot\nabla m\,-\,\k\,\Delta m\,=\,\alpha\,\o \\[1ex]
        u\,=\,-\,\nabla^\perp(-\Delta)^{-1}\o\,,
       \end{array}
\right.
\end{equation}
supplemented with the initial datum
\[
 \big(\o,m\big)_{|t=0}\,=\,\big(\o_0,m_0\big)\,,\qquad\qquad \mbox{ where }\qquad \o_0\,:=\,\curl(u_0)\,.
\]

Indeed, we have the following result, whose proof is classical but for which, for the sake of completeness, we also provide a sketch of the proof. We postpone a few remarks
on the statement at the end of the proof.
\begin{lemma} \label{l:equiv}
Let $u_0$ be a divergence-free vector field such that $\o_0\,:=\,\curl(u_0)$ belongs to the space $L^{p_0}(\R^2)\cap L^\infty(\R^2)$, with $p_0\in\,]1,2[\,$.
Let $T>0$ be a given time and assume that $\o\in L^\infty_\loc\big([0,T[\,;L^{p_0}(\R^2)\cap L^\infty(\R^2)\big)$ is a weak solution to the transport equation
\[
 \d_t\o\,+\,u\cdot\nabla\o\,=\,-\,\alpha\,\Delta m\,,\qquad u\,=\,-\,\nabla^\perp(-\Delta)^{-1}\o\,,
\]
related to the initial datum $\o_0$ and to some forcing term $m\in L^\infty_\loc\big([0,T[\,;L^{p_0}(\R^2)\cap L^\infty(\R^2)\big)$.

Then, there exists a tempered distribution $\nabla\Pi\in\mc S'\big([0,T[\,\times\R^2\big)$ such that the divergence-free vector field $u$ is a weak
solution to the equation
\begin{equation} \label{eq:u}
 \d_tu\,+\,(u\cdot\nabla)u\,+\,\nabla\Pi\,=\,-\alpha\,\nabla^\perp m\,,\qquad \div u\,=\,0\,,
\end{equation}
related to the initial datum $u_0$.
\end{lemma}

\begin{proof}
In order to simplify the presentation, let us deal with the integrability/regularity of the various quantity with respect to the space variable only. Indeed,
the time regularity ($L^\infty$ for both $\o$ and $m$) will not enter into play and the variable $t\geq0$ can be treated as a parameter in our argument.

To begin with, we observe that, thanks to the assumptions on $\o$ and to Proposition \ref{p:BS}, we deduce that
$u\in L^{q_0}\cap L^\infty$. Recall that $p_0<2<q_0$. Thus,  since $u$ is divergence-free and $\o\in L^{p_0}\cap L^\infty$, from the previous properties
we also infer that $u\,\in\,W^{1,q}$ for any $q\in[q_0,+\infty[\,$.

Next, we observe that the transport term in the equation of $\o$ can be written as
\begin{equation} \label{eq:writing-transp}
 u\cdot\nabla\o\,=\,\div\big(\o\,u\big)\,=\,\curl\big(\o\,u^\perp\big)\,,
\end{equation}
where, as usual in this paper, we have set $u^\perp\,=\,(-u^2,u^1)$ to be the rotation of angle $\pi/2$ of the vector $u$.
Therefore, in order to complete the proof, it is enough to apply the operator $-\nabla^\perp(-\Delta)^{-1}$ to both members of the equations: this is possible in the weak sense,
by testing the momentum equation on test-functions of the form $-(-\Delta)^{-1}\curl\psi$, with $\psi\in\mc D\big([0,T[\,\times\R^2;\R^2\big)$ be a smooth
compactly supported vector field over $\R^2$. Indeed, we observe that, for any two-dimensional vector field $v\in L^q(\R^2;\R^2)$, for some $q\in\,]1,+\infty[\,$, one has
\begin{align*}
-\nabla^\perp(-\Delta)^{-1}\curl v\,=\,v\,+\,\nabla\phi\,,\qquad\qquad &\mbox{ with }\qquad \nabla\phi\,\in\,L^q(\R^2;\R^2) \\
&\quad \mbox{ and }\qquad \nabla\phi\equiv0\quad \mbox{ if }\quad \div v \,=\,0\,.
\end{align*}
Using this property together with relation \eqref{eq:writing-transp} and the algebraic identity
\[
\o\,u^\perp\,=\,(u\cdot\nabla)u\,-\,\frac{1}{2}\,\nabla |u|^2
\]
in the case of space dimension $d=2$ (as in our setting), it is easy to deduce that $u$ satisfies the claimed equation, for a suitable
tempered distribution $\nabla\Pi$.

Before concluding, let us spend a few words about the time regularity of $\nabla\Pi$, because in principle the operator $(-\Delta)^{-1}$ may hide some terms with wild
dependence on the time variable. Actually, owing to \eqref{eq:u}, we see that this is not possible, as $\nabla\Pi$ must be a $W^{-1,\infty}$ distribution with respect to time.
Hence, we finally deduce that $\nabla\Pi\,\in\,\mc S'\big([0,T[\,\times\R^2\big)$, as claimed in the statement.
\end{proof}

%

For later use (especially in view of the uniqueness argument of Section \ref{s:uniqueness}), we now want to further analyse the regularity
of the pressure gradient appearing in the equation for $u$.

\begin{lemma} \label{l:pressure}
Let the integrability indices $p_0$ and $q_0$ be given as in Proposition \ref{p:BS}.
Given some time $T>0$, assume that equation \eqref{eq:u} holds true in the weak sense on $[0,T[\,\times\R^2$:
\[
 \d_tu\,+\,(u\cdot\nabla)u\,+\,\nabla\Pi\,=\,-\alpha\,\nabla^\perp m\,,\qquad \div u\,=\,0\,,
\]
with $u\in L^\infty_\loc\big([0,T[\,;L^{q_0}(\R^2)\cap LL(\R^2)\big)$ being a vector field such that $\o\,:=\,\curl(u)$ satisfies
$\o\in L^\infty_\loc\big([0,T[\,;L^{p_0}(\R^2)\cap L^\infty(\R^2)\big)$,
with $m\in L^\infty_\loc\big([0,T[\,;L^{p_0}(\R^2)\cap L^\infty(\R^2)\big)$ and with a suitable pressure term $\nabla\Pi\in \mc S'\big([0,T[\,\times\R^2\big)$.

Then, one has $\nabla\Pi\,\in\,L^\infty_\loc\big([0,T[\,;L^{p_0}(\R^2)\cap L^\infty(\R^2)\big)$.
\end{lemma}

\begin{proof}
We are going to essentially follow the analysis performed in Sections 3 and 5 of \cite{F_2025_Yud}. In addition, since all the quantities involved in the computations
are $L^\infty_\loc$ over the time interval $[0,T[\,$, let us treat the time as a parameter and focus our discussion just on space integrability/regularity.

We start by applying the divergence operator to equation \eqref{eq:u},
which simply consists in testing the equations on test-functions of the form $\nabla\vphi$, with $\vphi\in\mc D\big([0,T[\,\times\R^2\big)$. Since $\div u=0$,
this yields an elliptic equation for the pressure function:
\begin{equation} \label{eq:ell-p}
 -\,\Delta\Pi\,=\,\,\div\big((u\cdot\nabla)u\big)\,.
\end{equation}
By assumptions over $\o$ and the Calder\'on-Zygmund theory, we know that $\nabla u\in L^{p}(\R^2)$ for any $p\in[p_0,+\infty[\,$. Using that $u\in L^\infty(\R^2)$,
we easily deduce the property
\[
 (u\cdot\nabla)u\,\in\,\bigcap_{p_0\leq p<+\infty} L^p(\R^2)\,.
\]
Observe that, from equation \eqref{eq:ell-p}, one infers the following relation:
\[
 \nabla\Pi\,=\,\nabla(-\Delta)^{-1}\div\big((u\cdot\nabla)u\big)\,.
\]
Therefore, by Calder\'on-Zygmund theory again, we obtain that $\nabla\Pi\in L^{p}(\R^2)$ for any $p\in[p_0,+\infty[\,$.

On the other hand, by using the divergence-free condition over $u$, we can rewrite \eqref{eq:ell-p} as
\[
 -\,\Delta\Pi\,=\,\nabla u:\nabla u\,,
\]
where the right-hand side makes sense a.e. on $\R^2$ and where we have denoted by $A:B\,:=\,{\rm tr}(A\cdot\,^tB)$ the Frobenius product of two matrices $A$ and $B$.
Using again the property $\nabla u\in L^{p}(\R^2)$ for any $p\in[p_0,+\infty[\,$, we then conclude that
$-\Delta\Pi$ belongs to $L^q(\R^2)$ for any $q\in[1,+\infty[\,$. In particular, this yields that $\nabla\Pi \in W^{1,q}(\R^2)$ for any $q\in[p_0,+\infty[\,$.
Taking \tsl{e.g.} $q=q_0>2$, we infer that $\nabla\Pi\in L^\infty(\R^2)$ by Sobolev embeddings.
This concludes the proof of the lemma.
\end{proof}

%
 Should one suppose some integrability for the term $-\alpha\,\nabla^\perp m$, one would recover Lebesgue bounds for the term $\d_tu$ appearing in the equation
 for $u$ and thus prove that, actually, $u$ is a strong solution of the resulting equation. More precisely, one has the following statement.
%
 
 \begin{lemma} \label{l:d_tu}
Under the assumptions of Lemma \ref{l:pressure}, suppose in addition that $\nabla m$ belongs to $L^2_\loc\big([0,T[\,;L^2(\R^2)\big)$.

Then one has $\d_tu\,\in\,L^2_\loc\big([0,T[\,;L^2(\R^2)\big)$. In particular, $u$ is a strong solution to equation \eqref{eq:u}.
\end{lemma}

\begin{proof}
 The claim is an immediate consequence of 
Lemma \ref{l:pressure} and its proof,
in which we proved in particular that, under our assumptions,
both $(u\cdot\nabla)u$ and $\nabla\Pi$ belong to $L^\infty_\loc\big([0,T[\,;L^2(\R^2)\big)$.
Combining these properties with the hypothesis over $\nabla m$ finally yields the result.
\end{proof}

\subsection{An auxiliary unknown and global $L^p$ bounds} \label{ss:leb}
Let us come back to the reduced model \eqref{eq:2D-microp} for micropolar fluids in 2-D.
Thanks to Lemma \ref{l:equiv}, from now on we can safely focus on the study of its vorticity formulation, namely system \eqref{eq:o-m}.
The goal of this subsection is to prove the claimed global bounds on the $L^p$ norms of the solution. Recall that we are working in the context of \tsl{a priori} estimates,
so all the computations we are going to perform are only formal, so far.

The difficulty here relies on the following fact: propagating the $L^\infty$ norm of the vorticity, which seems to be a crucial condition for getting
uniqueness \cite{Yud_1963, A-B-C-DL-G-J-K},
requires a $L^\infty$ bound on $\Delta m$, but this looks out of reach from the parabolic equation for $m$, unless additional regularity is imposed on $\o$ itself. 
Inspired by \cite{H-K-R}, however, we introduce an auxiliary unknown, that is
\begin{equation} \label{eq:def_G}
 \Gamma\,:=\,\frac{1}{\alpha}\,\o\,+\,\frac{1}{\k}\,m\,.
\end{equation}
Combining the equation for $\o$ and the one for $m$, we easily see that $\Gamma$ satisfies the following transport equation:
\begin{equation} \label{eq:Gamma}
 \d_t\Gamma\,+\,u\cdot\nabla\Gamma\,=\,\frac{\alpha}{\k}\,\o\,=\,\frac{\alpha^2}{\k}\,\Gamma\,-\,\frac{\alpha^2}{\k^2}\,m\,.
\end{equation}
We observe that the equation for $\Gamma$ is a transport equation with non-zero right-hand side, which is suitable to close $L^p$-type estimates by a Gr\"onwall argument.
This is the key point to get global in time bounds at the level of regularity prescribed by Yudovich theory.
In what follows, we are going to derive those bounds.

\medbreak
We start by considering the transport equation \eqref{eq:Gamma} for $\Gamma$. Since $u$ is divergence-free, we immediately get,
for $p=p_0$ and $p=+\infty$ and for any $t\geq0$, the bound
\begin{align}
 \label{est:Gamma}
\left\|\Gamma(t)\right\|_{L^p}\,\leq\,\left\|\Gamma_0\right\|_{L^p}\,+\,C\int^t_0\Big(\left\|\Gamma(\t)\right\|_{L^p}\,+\,\left\|m(\t)\right\|_{L^p}\Big)\,\dd\t\,,
\end{align}
for a suitable constant $C>0$ depending only on $\alpha$ and $\k$. Here, with obvious notation, we have defined $\Gamma_0$ by the same formula in \eqref{eq:def_G}, where we
use the initial data $\o_0$ and $m_0$ instead of $\o$ and $m$, respectively.

Let us now switch our attention to the equation for the scalar field $m$. 
By \tsl{e.g.} applying a pointwise method (see for instance Section 3.1 of \cite{F-GB}) to follow the dynamics of the points of maximum and minimum of $m$, we can derive
the following $L^\infty$ bounds for it: for any $t\geq0$, we have
\begin{align}
 \label{est:m-inf}
\left\|m(t)\right\|_{L^\infty}\,\leq\,\left\|m_0\right\|_{L^\infty}\,+\,C\,\int^t_0\Big(\left\|\Gamma(\t)\right\|_{L^\infty}\,+\,\left\|m(\t)\right\|_{L^\infty}\Big)\,\dd\t\,,
\end{align}
where again the constant $C>0$ depends only on $\alpha$ and $\k$.
Next, we multiply the equation for $m$ by the quantity $|m|^{p_0-2}\,m$ and we integrate over $\R^2$: after noticing that
\[
 \forall\,p\in\,]1,+\infty[\,,\qquad \int_{\R^2}(-\Delta m)\,|m|^{p-2}\,m\,\dx\,=\,(p-1)\int_{\R^2}|m|^{p-2}\,|\nabla m|^2\,\dx\,\geq\,0\,,
\]
we obtain the following estimate, for a constant $C=C(\alpha,\k)>0$ depending only on the quantities in the brackets:
\begin{align}
 \label{est:m-p_0}
\left\|m(t)\right\|_{L^{p_0}}\,\leq\,\left\|m_0\right\|_{L^{p_0}}\,+\,C\,\int^t_0\Big(\left\|\Gamma(\t)\right\|_{L^{p_0}}\,+\,\left\|m(\t)\right\|_{L^{p_0}}\Big)\,\dd\t\,.
\end{align}
Performing similar computations for $p=2$, we find instead an additional control on the gradient of $m$: 
for a new constant $C=C(\alpha,\k)>0$, we get
\begin{align}
 \label{est:m-2_provv}
\left\|m(t)\right\|^2_{L^{2}}\,+\,\int^t_0\left\|\nabla m(\t)\right\|^2_{L^2}\,\dd\t\,\leq\,
\left\|m_0\right\|^2_{L^{2}}\,+\,C\,\int^t_0\Big(\left\|\Gamma(\t)\right\|^2_{L^{2}}\,+\,\left\|m(\t)\right\|^2_{L^{2}}\Big)\,\dd\t\,.
\end{align}

Summing up inequalities \eqref{est:Gamma}, \eqref{est:m-inf} and \eqref{est:m-p_0} and applying the Gr\"onwall lemma, we finally deduce the following
global bound for the Lebesgue norms of $\Gamma$ and $m$: for a constant $C=C(\alpha,\k)>0$ as above, one has
\begin{align}
\label{est:G-m_final}
 \forall\,t\geq0\,,\qquad \left\|\Gamma(t)\right\|_{L^{p_0}\cap L^\infty}\,+\,\left\|m(t)\right\|_{L^{p_0}\cap L^\infty}\,\leq\,
 C\,\Big(\left\|\o_0\right\|_{L^{p_0}\cap L^\infty}\,+\,\left\|m_0\right\|_{L^{p_0}\cap L^\infty}\Big)\,e^{C\,t}\,,
\end{align}
where we have set $\|f\|_{X\cap Y}\,:=\,\|f\|_X\,+\,\|f\|_Y$ and where we have also used the expression of $\Gamma_0$ in terms of $\o_0$ and $m_0$.
Using the just proven \eqref{est:G-m_final} into \eqref{est:m-2_provv} also yields
\begin{align}
 \label{est:m-2_final}
\left\|m(t)\right\|^2_{L^{2}}\,+\,\int^t_0\left\|\nabla m(\t)\right\|^2_{L^2}\,\dd\t\,\leq\,C\,
\left(\left\|\o_0\right\|^2_{L^{p_0}\cap L^\infty}\,+\,\left\|m_0\right\|^2_{L^{p_0}\cap L^\infty}\right)\,e^{C\,t}\,,
\end{align}
where the constant $C=C(\alpha,\k)>0$ is as before. We observe that the above $L^2$ estimate for $m$ is not really needed for the proof of existence,
whereas it will play an important role for the uniqueness argument, in order to mimic the original proof of Yudovich (see \cite{Yud_1963}, see also
Chapter 8 of \cite{Maj-Bert}). On the other hand, we remark that this estimate comes for free from our assumptions.

Of course, recovering $\o$ in terms of $\Gamma$ and $m$ from relation \eqref{eq:def_G} yields a similar bound also for the vorticity:
\begin{align}
\label{est:o_final}
 \forall\,t\geq0\,,\qquad \left\|\o(t)\right\|_{L^{p_0}\cap L^\infty}\,\leq\,
C\, \Big(\left\|\o_0\right\|_{L^{p_0}\cap L^\infty}\,+\,\left\|m_0\right\|_{L^{p_0}\cap L^\infty}\Big)\,e^{C\,t}\,,
\end{align}
where $C=C(\alpha,\k)>0$ only depends on the quantities inside the brackets, as before.

Finally, using Proposition \ref{p:BS}, we deduce similar estimates also for the velocity field:
\begin{align}
\label{est:u_final}
 \forall\,t\geq0\,,\qquad \left\|u(t)\right\|_{L^{q_0}\cap LL}\,\leq\,
C\, \Big(\left\|\o_0\right\|_{L^{p_0}\cap L^\infty}\,+\,\left\|m_0\right\|_{L^{p_0}\cap L^\infty}\Big)\,e^{C\,t}\,,
\end{align}
where we recall that $q_0\in\,]2,+\infty[\,$ is defined by the formula $1/q_0\,=\,1/p_0\,-\,1/2$ and where the constant $C>0$ is as above.

\section{Proof of existence} \label{s:existence}
In this section, we prove the existence of global solutions at the level of regularity claimed in Theorem \ref{th:global-wp}.
Owing to Lemma \ref{l:equiv}, it is enough to construct solutions $\big(\o,m\big)$ to the vorticity formulation \eqref{eq:o-m} of our system. For this, we will
proceed in a rather classical way: after smoothing out the initial datum, we will construct a sequence of smooth solutions of the equations related to
the family of regularised initial data.
After showing suitable uniform bounds for the family of smooth approximate solutions,
we will then complete the proof of the existence by passing to the limit in the approximation parameter, by means of a weak compactness method.

For convenience of notation, throughout this section we will adopt the following convention. 
Given a sequence of functions $\big(f_n\big)_{n\in\N}\,\subset\,\mf B$, where $\mf B$ is a Banach space,
we will write $\big(f_n\big)_{n\in\N}\sqsubset \mf B$ if there exists a constant $C>0$ such that $\|f_n\|_{\mf B}\leq C$ for all $n\in\N$.

\subsection{Regularisation of the initial datum} \label{ss:regul}

Let $\big(u_0,m_0\big)$ be the initial datum fixed in the statement of Theorem \ref{th:global-wp}, so that
\[
 \o_0\,,\,m_0\;\in\,L^{p_0}(\R^2)\cap L^\infty(\R^2)\,.
\]
Recall that $\o_0\,:=\,\curl(u_0)$ is the initial vorticity and that $p_0\in\,]1,2[\,$ by assumption.
Then, owing to Proposition \ref{p:BS}, we also know that
\[
u_0\,\in\,L^{q_0}(\R^2) \cap L^{\infty}(\R^2) \cap LL(\R^2)\,,
\]  
where $q_0\in\,]2,+\infty[\,$ is defined by the formula $1/q_0\,=\,1/p_0\,-\,1/2$. Observe that, by Definition \ref{def:LL}, one has
$LL\hookrightarrow L^\infty$, however we write both spaces here above, to stress the property $u_0\in L^\infty$.

By taking \tsl{e.g.} a family of standard mollifiers, we smooth out the initial datum, in such a way that we obtain a sequence of smooth vector fields  
\(\big(u_{0,n}\big)_{n\in\N}\)   and a sequence of smooth functions  \(\big(m_{0,n}\big)_{n\in\N}\)  such that  
\begin{align*}
& \forall\,n\in\N\,,\quad \div u_{0,n}\,=\,0\,, \\
& 
\quad \big(u_{0,n}\big)_{n\in\N}\,\sqsubset\,L^{q_0}(\R^2) \cap L^{\infty}(\R^2) \cap LL(\R^2) \\
&\qquad \mbox{ and }\qquad \big(\o_{0,n}\big)_{n\in\N}\,,\,\big(m_{0,n}\big)_{n\in\N}\;\sqsubset\,L^{p_0}(\R^2)\cap L^\infty(\R^2)\,,
\end{align*}
where, with obvious notation, we have set $\o_{0,n}\,=\,\curl(u_{0,n})$. Notice that, since derivatives commute with convolutions, the function $\o_{0,n}$
is the regularisation, \tsl{via} the mollifying kernel, of the initial vorticity $\o_0$.

More precisely, let us observe that there exists an absolute constant $C>0$, depending only on the properties of the mollifying kernel, such that
\begin{align}
 \label{unif_est:data}
\forall\,n\in\N\,,\qquad 
 &\left\|\o_{0,n}\right\|_{L^{p_0}\cap L^\infty}\,\leq\,C\, \left\|\o_{0}\right\|_{L^{p_0}\cap L^\infty} \\
 \nonumber
 &\quad \mbox{ and }\qquad
\left\|m_{0,n}\right\|_{L^{p_0}\cap L^\infty}\,\leq\,C\,\left\|m_{0}\right\|_{L^{p_0}\cap L^\infty}\,.
\end{align}

In addition, we have the following convergence properties, where we stress that the convergence holds true with respect to the strong topology:
in the limit for $n\to+\infty$, one has
\begin{align*}
& u_{0,n}\, \longrightarrow\,u_0 \qquad\qquad  \text{ in } \qquad L^{q_0}(\R^2) \cap L^\infty_{\loc}(\R^2) \\
& \quad \mbox{ and }\qquad \big(\omega_{0,n}, m_{0,n}\big)\,\longrightarrow\,\big(\o_0, m_0\big) \qquad\qquad \text{ in } \qquad L^{p}(\R^2)\,,
\ \forall\,p\in[p_0,+\infty[\,\,.
\end{align*}   
In fact, one could easily prove that $u_{0,n}\, \longrightarrow\,u_0$ in $L^\infty(\R^2)$, with no localisation on compact sets, owing
to the property $u_0\in LL(\R^2)$. However, this stronger convergence is not needed in our argument.

For later use, we also introduce the approximated initial datum for the variable $\Gamma$, recall definition \eqref{eq:def_G}: we set
\begin{equation} \label{eq:def_G_0-n}
\Gamma_{0,n}(x)\,:=\,\frac{1}{\alpha}\, \omega_{0,n}(x)\,+\,\frac{1}{\kappa}\,m_{0,n}(x)\,.
\end{equation}
It goes without saying that $\big(\Gamma_{0,n}\big)_{n\in\N}\,\sqsubset\, L^{p_0}(\R^2)\cap L^\infty(\R^2)$ and that the strong convergence
$\Gamma_{0,n}\,\longrightarrow\, \Gamma_0$ holds true, when $n\to+\infty$, in the topology of $L^{p}(\R^2)$, for any $p\in[p_0,+\infty[\,$.

\subsection{The sequence of smooth approximate solutions} \label{ss:smooth-sol}
We now construct smooth approximate solutions associated to the sequence of regularised initial data.
As a matter of fact, thanks to the results of \cite{F-FD-MM}, for each $n\in\N$ we can solve
system \eqref{eq:2D-microp} with initial datum $\big(u_{0,n},m_{0,n}\big)$ globally in time.

More precisely, for any $n\in\N$, we have that
\[
 u_{0,n}\quad \mbox{ and }\quad m_{0,n}\qquad\quad \mbox{ belong to }\qquad \bigcap_{k\geq 0}\mc C^k_b(\R^2)\,,
\]
where we have denoted by $\mc C^k_b$ the space of $\mc C^k$ functions which are bounded with all their derivatives up to order $k$.
In addition, we have $\div u_{0,n}\,=\,0$ and both $\o_{0,n}$ and $m_{0,n}$ belong to $L^{p_0}(\R^2)$.
All the assumptions of Theorem 1.2 of \cite{F-FD-MM} are then satisfied. For the reader's convenience, let us recall a simplified version of that
statement here below.

\begin{thm} \label{th:smooth}
Let $k\,\in\,\N$ such that $k\,\geq\,2$. Take an initial datum $\big(u_0,m_0\big)$ such that
\[
 \div u_0\,=\,0\,,\qquad u_0\,\in\,\mc C^k_b(\R^2)\qquad \mbox{ and }\qquad m_0\,\in\,\mc C^{k-1}_b(\R^2)\,.
\]
Assume moreover that there exists $p_0\in\,]1,2[\,$ such that both $\o_0$ and $m_0$ belong to $L^{p_0}(\R^2)$,
where we have defined  $\o_0\,:=\,\curl(u_0)$ to be the vorticity of the initial velocity $u_0$.

Then, there exists a unique global in time solution $\big(u,m\big)$ to system \eqref{eq:2D-microp} related to the initial datum $\big(u_0,m_0\big)$, such that
\begin{align*}
& u\,\in\,L^\infty_\loc\big(\R_+;\mc C^k_b(\R^2)\big)\,\cap\,\mc C\big(\R_+;\mc C^{k-1}_b(\R^2)\big)\,, 
\qquad \mbox{ with }\quad \o\,\in\, \mc C\big(\R_+;L^{p_0}(\R^2)\big)\,, \\
&\qquad \mbox{ and }\qquad\qquad
m\,\in\,L^\infty_\loc\big(\R_+;\mc C^{k-1}_b(\R^2)\big)\,\cap\,\mc C\big(\R_+;\mc C^{k-2}_b(\R^2)\big)\,\cap\,\mc C\big(\R_+;L^{p_0}(\R^2)\big)\,.
\end{align*}
\end{thm}

Therefore, owing to Theorem \ref{th:smooth}, we construct the smooth approximate solution $\big(u^n,m^n\big)$ as the unique solution
to the system \eqref{eq:2D-microp} related to $\big(u_{0,n},m_{0,n}\big)$.
Notice that this solution is defined globally in time. We also point out that the above statement can be applied with any value of $k\geq2$, so
the $n$-th solution is indeed smooth, as claimed.

The next step consists in deriving uniform bound for the sequence $\big(u^n,m^n\big)_{n\in\N}$ that we have just constructed above.
For this, it is useful to write down the explicit equations solved by this sequence: after denoting $\o^n\,:=\,\curl(u^n)$ the vorticity
related to the velocity field $u^n$, one has
\begin{equation} \label{eq:microp_n}
\left\{\begin{array}{l}
        \d_tu^n\,+\,(u^n\cdot\nabla)u^n\,+\,\nabla\Pi^n\,=\,-\,\alpha\,\nabla^\perp m^n \\[1ex]
        \d_tm^n\,+\,u^n\cdot\nabla m^n\,-\,\k\,\Delta m^n\,=\,\alpha\,\o^n \\[1ex]
        \div u^n\,=\,0\,,
       \end{array}
\right.
\end{equation}
related to the initial datum
\[
 \big(u_n,m_n\big)_{|t=0}\,=\,\big(u_{0,n}, m_{0,n}\big)\,,
\]
where the couple $\big(u_{0,n}, m_{0,n}\big)$ is the couple of smooth initial data defined in Subsection \ref{ss:regul}.

We also notice that the vorticity $\o^n\,=\,\curl(u^n)$ solves the equation
\begin{equation}
\label{eq:def_o_n}
    \d_t\o^n\,+\,u^n\cdot\nabla\o^n\,=\,-\,\alpha\,\Delta m^n\,,\qquad\quad \mbox{ with }\quad \o^n_{|t=0}\,=\,\o_{0,n}\,,
\end{equation}
and that, after defining $\Gamma_{0,n}$ as in \eqref{eq:def_G_0-n} and $\Gamma^n$ by the analogous formula related
to $m^n$ and $\o^n$, keep in mind \eqref{eq:def_G}, one has that $\Gamma^n$ satisfies the transport equation
\begin{equation}
\label{eq:def_G_n}
\d_t\Gamma^{n}\,+\,u^n\cdot\nabla\Gamma^{n}\,=\,\frac{\alpha}{\k}\,\o^n\,,
\end{equation}
related to the initial datum $\Gamma_{0,n}$.

\subsection{Uniform bounds for the approximate solutions} \label{ss:unif-bounds}

Let us establish uniform bounds for the family of smooth approximate solutions constructed in the previous subsection.
More precisely, here we are going to prove the following statement.

\begin{lemma} \label{l:uniform}
There exist universal positive constants $C_0>0$ and $K_0>0$ such that, for any time $T>0$ fixed, one has the following estimate:
for any integer $n\in\N$, one has
\begin{equation} \label{est:unif}
\left\|\big(\o^n,m^n\big)\right\|_{L^\infty_T(L^{p_0}\cap L^\infty)}\,+\,\k\,\left\|\nabla m^n\right\|_{L^2_T(L^2)}\,\leq\,
C_0\,\left\|\big(\o_0,m_0\big)\right\|_{L^{p_0}\cap L^\infty}\,e^{K_0\,T}\,,
\end{equation}
where, for simplicity of notation, we have set $\|(f,g)\|_X\,=\,\|f\|_X\,+\,\|g\|_X$.
\end{lemma}

\begin{proof}
Formula \eqref{est:unif} simply follows from the computations performed in Subsection \ref{ss:leb}, after observing that, for each $n\in\N$ fixed, we have enough
regularity to carry out those computations rigorously.

First of all, similarly to what done in Subsection \ref{ss:leb},
standard $L^p$ estimates performed on the second equation of \eqref{eq:microp_n} and on \eqref{eq:def_G_n} yield, for any time $t\geq0$, the inequality
\begin{align*}
&\left\|\Gamma^{n}(t)\right\|_{L^{p_0}\cap L^\infty}\,+\,\left\|m^{n}(t)\right\|_{L^{p_0}\cap L^\infty} \\
&\qquad \leq\,
C\,\Big(\left\|\o_{0,n}\right\|_{L^{p_0}\cap L^\infty}\,+\,\left\|m_{0,n}\right\|_{L^{p_0}\cap L^\infty}\Big)\,+\,
C\,\int^t_0\left\|\o^n\right\|_{L^{p_0}\cap L^\infty}\,\dd \t\,.
\end{align*}
Now, using the relation \eqref{eq:def_G} linking $\o^n$, $m^n$ and $\Gamma^n$, together with the inequalities in \eqref{unif_est:data},
from that estimate we infer
\begin{align*}
&\left\|\o^{n}(t)\right\|_{L^{p_0}\cap L^\infty}\,+\,\left\|m^{n}(t)\right\|_{L^{p_0}\cap L^\infty}\,\leq\,
C_1\,\left\|\big(\o_{0},m_0\big)\right\|_{L^{p_0}\cap L^\infty}\,+\,C_2\,\int^t_0\left\|\o^n\right\|_{L^{p_0}\cap L^\infty}\,\dd \t\,,
\end{align*}
for two universal constants $C_{1,2}>0$ only depending on the parameters of the problem, but not on the data, nor on the solution, nor on $n\geq1$.
An application of the Gr\"onwall lemma thus yields
\begin{align}
\label{est:m-o_unif}
\forall\,t\geq0\,,\qquad\quad
\left\|\o^{n}(t)\right\|_{L^{p_0}\cap L^\infty}\,+\,\left\|m^{n}(t)\right\|_{L^{p_0}\cap L^\infty}\,\leq\,
C_1\,\left\|\big(\o_{0},m_0\big)\right\|_{L^{p_0}\cap L^\infty}\;e^{C_2\,t}\,.
\end{align}
This completes the proof of the first part of inequality \eqref{est:unif}, with the possibility of taking
\begin{equation} \label{cond_const:1}
  C_0\,\geq\,C_1\qquad\qquad \mbox{ and }\qquad\qquad K_0\,\geq\,C_2\,.
\end{equation}

We are now going to exhibit a uniform bound for the terms $\nabla m^n$, for $n\in\N$.
Performing an energy estimate on the equation for $m^n$ in system \eqref{eq:microp_n} and applying inequality \eqref{est:m-o_unif},
by simple computations we find, for any $t\geq0$, the estimate
\begin{align*}
 \frac{1}{2}\,\left\|m^{n}(t)\right\|_{L^2}^2\,+\,\k\,\int^t_0\left\|\nabla m^{n}(\t)\right\|_{L^2}^2\,\dd\t\,&\leq\,
\frac{1}{2}\,\left\|m_{0,n}\right\|^2_{L^2}\,+\,\int^t_0\left\|\o^n\right\|_{L^2}\,\left\|m^{n}\right\|_{L^2}\,\dd\t \\
 &\leq\,\left(\frac{C}{2}\,+\,\frac{C_1^2}{2\,C_2}\right)\,\left\|\big(\o_{0},m_0\big)\right\|_{L^{p_0}\cap L^\infty}^2\,e^{2\,C_2\,t}\,,
\end{align*}
where the constant $C$ comes from the bound \eqref{unif_est:data} for the regularisation of the initial data.
Therefore, if we further choose $C_0$ and $K_0$ such that, besides \eqref{cond_const:1}, they also satisfy the conditions
\[
 \frac{C}{2}\,+\,\frac{C_1^2}{2\,C_2}\,\leq\,C_0^2\,,
\]
then we deduce that
\[
 \k\,\left\|\nabla m^{n}\right\|_{L^2_T(L^2)}\,\leq\,C_0\,\left\|\big(\o_{0},m_0\big)\right\|_{L^{p_0}\cap L^\infty}\,e^{K_0\,t}\,,
\]
thus completing the proof of the full bound \eqref{est:unif}. This ends the proof of the lemma.
\end{proof}

As a direct consequence of Lemma \ref{l:uniform}, also combined with Proposition \ref{p:BS}, we immediately get the following result.

\begin{cor} \label{c:unif}
The sequences $\big(\o^n\big)_{n\in\N}$ and $\big(m^n\big)_{n\in\N}$ are included in $L^\infty_\loc\big(\R_+;L^{p_0}(\R^2)\cap L^\infty(\R^2)\big)$, while
the sequence $\big(u^n\big)_{n\in\N}$ is included in the space $L^\infty_\loc\big(\R_+;L^{q_0}(\R^2)\cap LL(\R^2)\big)$ and the sequence
$\big(\nabla m^n\big)_{n\in\N}$ is included in $L^2_\loc\big(\R_+;L^{2}(\R^2)\big)$.

In addition, for any time $T>0$ fixed, we have that
\begin{align*}
& \big(\o^n\big)_{n\in\N}\,,\, \big(m^n\big)_{n\in\N}\;\sqsubset\,L^\infty\big([0,T];L^{p_0}(\R^2)\cap L^\infty(\R^2)\big)\,, \\
&\quad \big(\nabla m^n\big)_{n\in\N}\,\sqsubset\,L^2\big([0,T];L^{2}(\R^2)\big)
\qquad \mbox{ and }\qquad \big(u^n\big)_{n\in\N}\,\sqsubset\,L^\infty\big([0,T];L^{q_0}(\R^2)\cap LL(\R^2)\big)\,.
\end{align*}
\end{cor}

From the previous corollary we will deduce weak convergence properties (up to a suitable extraction) for the family of approximate solutions
$\big(u^n,m^n\big)_{n\in\N}$. However, in order to pass to the limit in the non-linear terms of the equations, some strong convergence is needed; in turn,
this requires to find suitable bounds for the time derivative of the approximate solutions, which will be established in the following statement.

\begin{lemma} \label{l:d_t-o^n}
 The sequences $\big(\d_tm^n\big)_{n\in\N}$ and $\big(\d_t\o^n\big)_{n\in\N}$ are bounded in $L^\infty_\loc\big(\R_+;W^{-2,p_0}(\R^2)\cap W^{-2,\infty}(\R^2)\big)$
 and in $L^2_\loc\big(\R_+;H^{-1}(\R^2)\big)$.
In particular, for any time $T>0$ fixed, one has 
\[
\big(\d_tm^n\big)_{n\in\N}\,,\, \big(\d_t\o^n\big)_{n\in\N}\;\sqsubset\,L^\infty\big([0,T];W^{-2,p_0}(\R^2)\cap W^{-2,\infty}(\R^2)\big)\,\cap\,
L^2\big([0,T];H^{-1}(\R^2)\big)\,.
\]
\end{lemma}

\begin{proof}
We start by considering, for any $n\in\N$, the equation for $\o^{n}$, which is given by formula \eqref{eq:def_o_n}.
Since $u^n$ is divergence-free, this equation can also be written as
\[
 \d_t\o^{n}\,=\,\,-\,\alpha\,\Delta m^{n}\,-\,\div\big(\o^{n}\,u^n\big)\,.
\]
Now, we use the bounds established in Corollary \ref{c:unif}. First of all, we observe that
\[
 \big(\Delta m^{n}\big)_{n\in\N}\,\subset\,L^\infty_\loc\big(\R_+;W^{-2,p_0}(\R^2)\cap W^{-2,\infty}(\R^2)\big)\,\cap\,L^2_\loc\big(\R_+;H^{-1}(\R^2)\big)\,,
\]
with uniform bounds whenever we restrict to any finite time interval $[0,T]$.
Then, we use the uniform bounds on $\big(\o^n\big)_{n\in\N}$ and $\big(u^n\big)_{n\in\N}$ to get
\[
 \big(\o^{n}\,u^n\big)_{n\in\N}\,\subset\,L^\infty_\loc\big(\R_+;L^a(\R^2)\cap L^\infty(\R^2)\big)\,,\qquad\quad \mbox{ with }\qquad
\frac{1}{a}\,:=\,\min\left\{1\,,\,\frac{1}{p_0}\,+\,\frac{1}{q_0}\right\}\,.
\]
Notice that $p_0\geq a$. Therefore, with those properties at hand, the claimed boundedness results for $\big(\d_t\o^n\big)_{n\in\N}$ easily follow.

Finally, the analysis of the terms $\d_tm^{n}$ relies on the study of the second equation in \eqref{eq:microp_n} and on the use of analogous arguments.
We omit the details here.
\end{proof}

\subsection{Convergence: end of the proof} \label{ss:convergence}
Thanks to the uniform boundedness properties established in the previous part, we can pass to the limit in the weak formulation of the equations,
thus completing the proof of the existence of a solution claimed in Theorem \ref{th:global-wp}.

We will make use of weak and strong convergence properties following from Corollary \ref{c:unif} and Lemma \ref{l:d_t-o^n}. We collect them in the next
statements. We start with a weak convergence result.
\begin{prop} \label{p:converg}
There exist limit profiles $\o$ and $m$, belonging to the space $L^\infty_\loc\big(\R_+;L^{p_0}(\R^2)\cap L^\infty(\R^2)\big)$, such that,
up to the extraction of a suitable subsequence, one has the weak-$*$ convergences
\[
\o^n\,\stackrel{*}{\rightharpoonup}\,\o\qquad \mbox{ and }\qquad m^n\,\stackrel{*}{\rightharpoonup}\,m\qquad\qquad 
\mbox{ in } \quad L^\infty_\loc\big(\R_+;L^{p_0}(\R^2)\cap L^\infty(\R^2)\big)\,.
\]
One moreover has that $m\,\in\,L^2_\loc\big(\R_+;H^1(\R^2)\big)$ and
\[
 \nabla m^n\,\rightharpoonup\,\nabla m\qquad\qquad \mbox{ in }\qquad L^2_\loc\big(\R_+;L^2(\R^2)\big)\,.
\]

In addition, defined $u\,:=\,-\,\nabla^\perp(-\Delta)^{-1}\o$ the velocity field related to $\o$ by the Biot-Savart law \eqref{eq:BS}, one has
$u\,\in\,L^\infty_\loc\big(\R_+;L^{q_0}(\R^2)\cap LL(\R^2)\big)$, together with the convergence
\[
 u^n\,\stackrel{*}{\rightharpoonup}\,u\qquad \mbox{ in }\quad L^\infty_\loc\big(\R_+;L^{q_0}(\R^2)\cap L^\infty(\R^2)\big)\,.
\]
\end{prop}

\begin{proof}
 The existence of the profiles $\o$ and $m$ in the claimed space, together with the weak-$*$ convergence properties, are direct consequences of
Corollary \ref{c:unif}, combined with the Banach-Alaoglu theorem and a diagonal process (needed when considering increasing time intervals covering the whole $\R_+$).

Owing to the characteristics of $\o$ and Proposition \ref{p:BS}, it is also plain to establish the claimed property
$u\in L^\infty_\loc\big(\R_+;L^{q_0}(\R^2)\cap LL(\R^2)\big)$, together with the weak-$*$ convergence of the sequence $\big(u^n\big)_{n\in\N}$ in that space
(without a further extraction, by uniqueness of the weak limit), thanks to the linearity of the Biot-Savart law.
\end{proof}

Let us now show the strong convergence (in a suitable space) of the sequence of the velocity fields.
Observe that, owing to Corollary \ref{c:unif} and Lemma \ref{l:d_t-o^n},
it would be easy to deduce strong convergence of $\big(m^n\big)_{n\in\N}$ towards $m$ in \tsl{e.g.} $L^2_\loc\big(\R_+\times\R^2\big)$; however, since 
passing to the limit in the vorticity equation requires strong convergence of the velocity fields anyway, we focus
on the strong convergence of the $u^n$'s only.

\begin{prop} \label{p:conv_strong}
There exists $\wtilde p\in\,]2,+\infty[\,$ such that, up to a further extraction (omitted to be noted here), 
for any time $T>0$ and any compact set $K\subset\R^2$ fixed, one has the strong convergence
\[
 u^n\,\longrightarrow\,u\qquad\qquad \mbox{ in }\qquad \mc C\big([0,T];L^{\wtilde p}(K)\big)
\]
in the limit $n\to +\infty$.
\end{prop}

\begin{proof}
Let $T>0$ be a given time, which can be chosen arbitrarily but which will be kept fixed throughout the argument below.

To derive the sought compactness properties for $\big(u^n\big)_{n\in\N}$, we first observe that, owing to Corollary \ref{c:unif} and the fact that
$p_0<2<q_0<+\infty$, one has 
\begin{equation} \label{ub:u^n-space}
\big(u^n\big)_{n\in\N}\,\sqsubset\,L^\infty\big([0,T];W^{1,q}(\R^2)\big)\qquad\quad \mbox{ for any }\quad q\in[q_0,+\infty[\;.
\end{equation}

Next, we claim that there exists $\wtilde p\in\,]2,+\infty[\,$ such that one also has the property
\begin{equation} \label{ub:d_tu^n}
\big(\d_tu^n\big)_{n\in\N}\,\sqsubset\,L^\infty\big([0,T];W^{-2,\wtilde p}(\R^2)\big)\,.
\end{equation}
To prove this claim, we are going to use Lemma \ref{l:d_t-o^n}; as the time variable will play no special role, let us simply focus on space
regularity in our discussion.

Assume for a while that $1<p_0\leq 4/3$: in this case, we take $\wtilde p\,=\,p_0'$, where, for any $p\in[1,+\infty]$,
we denote by $p'$ its conjugate Lebesgue exponent (namely, $1/p'\,=\,1\,-\,1/p$).
As a matter of fact, using the Biot-Savart law \eqref{eq:BS},
for a vector field $\psi\in W^{2,p_0}(\R^2)$, we can write (in the sense of distributions) the duality product
\begin{align*}
\left\lan\d_tu^n\,,\,\psi\right\ran\,&=\,\left\lan-\d_t\nabla^\perp(-\Delta)^{-1}\o^n\,,\,\psi\right\ran \\
&=\,\left\lan\d_t\o^n\,,\,-(-\Delta)^{-1}\curl(\psi)\right\ran\,.
\end{align*}
Now, since $\psi\in W^{2,p_0}(\R^2)\,\hookrightarrow\,W^{1,q_0}(\R^2)$, with $1/q_0\,=\,1/p_0\,-\,1/2$ (as given in Proposition \ref{p:BS}),
we notice that $-(-\Delta)^{-1}\curl(\psi)$ belongs to $W^{2,q_0}(\R^2)$. By our assumption
$1<p_0\leq 4/3$, we have that $2<q_0\leq p_0'$, so the term in the last line of the previous series of
equalities is well-defined owing to Lemma \ref{l:d_t-o^n}. Thus, in this case \eqref{ub:d_tu^n} is satisfied with $\wtilde p=p_0'$.

Let us now focus on the proof of the claim \eqref{ub:d_tu^n} for $4/3<p_0<2$. Observe that, in this case, one has $2<p_0'<4$. 
Next, we remark that, by Sobolev embeddings in $\R^2$ and similarly to what just done above, for any index $p_1\in\,]1,4/3[\,$ one has
$W^{2,p_1}(\R^2)\,\hookrightarrow\,W^{1,q_1}(\R^2)$,
where $q_1$ is given by the formula $1/q_1\,=\,1/p_1\,-\,1/2$; remark that $q_1\in\,]2,4[\,$, precisely as $p_0'$.
At the same time, we notice that, owing to Proposition \ref{p:BS}, for a vector field
$\psi\in L^{p_1}(\R^2)$, one has $-(-\Delta)^{-1}\curl(\psi)\in L^{q_1}(\R^2)$, where $q_1$ is precisely the index fixed above. It is then a matter of fixing
$p_1$ so that $q_1 = p_0'$, which in turn reduces to the condition
\[
 \frac{1}{p_1}\,=\,\frac{3}{2}\,-\,\frac{1}{p_0}\,.
\]
Indeed,  with this choice we see that, for $\psi\in W^{2,p_1}(\R^2)$, one has $-(-\Delta)^{-1}\curl(\psi)\in W^{2,p_0'}$, so the duality pairing
$\left\lan\d_t\o^n\,,\,-(-\Delta)^{-1}\curl(\psi)\right\ran$ makes sense owing to Lemma \ref{l:d_t-o^n} again.
Then, the claim is proven by choosing $\wtilde p = p_1' = q_0$, where, as usual in this work, $q_0$ is the index related to $p_0$ through Proposition \ref{p:BS}.

All in all, combining \eqref{ub:u^n-space} and \eqref{ub:d_tu^n} together and keeping in mind the above choice of the index $\wtilde p$,
we find either that
\begin{align*}
\mbox{ if }\ 1\,<\,p_0\,\leq\,\frac{4}{3}\,,\qquad\qquad \big(u^n\big)_{n\in\N}\,\sqsubset\,L^\infty\big([0,T];W^{1,p_0'}(\R^2)\big)\,\cap\,
W^{1,\infty}\big([0,T];W^{-2,p_0'}\big),
\end{align*}
because $q_0\leq p_0'$ when $p_0$ is in that range, or that
\begin{align*}
\mbox{ if }\ \frac{4}{3}\,<\,p_0\,<\,2\,,\qquad\qquad \big(u^n\big)_{n\in\N}\,\sqsubset\,L^\infty\big([0,T];W^{1,q_0}(\R^2)\big)\,\cap\,
W^{1,\infty}\big([0,T];W^{-2,q_0}\big)\,,
\end{align*}
since, as remarked above, one has $\wtilde p = q_0$ for $p_0\in\,]4/3,2[\,$. 
Observe that, in these last two uniform embedding properties (more precisely, when speaking about the Lipschitz continuity in time of the sequence of the velocity fields),
we have implicitly used that $u_0\in L^{q_0}\cap L^\infty$, according to Proposition \ref{p:BS}. Indeed, in particular $u_0\in W^{-2,q_0}$;
in addition, when $1<p_0<4/3$, one has $q_0\leq p_0'$, so $u_0$ belongs also to $L^{p_0'}$, hence to $W^{-2,p_0'}$. This implies, of course, that the sequence
$\big(u_{0,n}\big)_{n\in\N}$ is uniformly bounded in those spaces.

This having been clarified, the sought strong convergence of the velocity fields (up to extraction of a subsequence) follows from the compact embedding
$W^{1,p}(K)\,\hookrightarrow\hookrightarrow\,L^p(K)$ for any compact set $K\subset\R^2$ and $p\in\,]1,+\infty[\,$, the application of the Ascoli-Arzel\`a theorem
and a standard use of a diagonalisation argument.
\end{proof}

With the weak convergence properties from Proposition \ref{p:converg} and the strong convergence of the velocity fields from Proposition \ref{p:conv_strong},
it is an easy matter to pass to the limit $n \to +\infty$ in the weak formulation of equations \eqref{eq:microp_n}.
This argument in turn proves
that the couple $\big(u,m\big)$, identified in Proposition \ref{p:converg}, solves system \eqref{eq:2D-microp} with initial datum the given
$\big(u_0,m_0\big)$.

We omit the details here. We simply point out that the weak formulation of the equation of $u^n$ requires to use divergence-free test-functions
$\psi\in\mc D\big(\R_+\times\R^2\big)$; then, the pressure gradient $\nabla\Pi$ can be recovered by standard distribution theory and the use of Lemma \ref{l:pressure}.
We also observe that, alternatively, using the same Lemma \ref{l:pressure}, one could instead derive uniform bounds for the family of pressure gradients
$\big(\nabla\Pi^n\big)_{n\in\N}$ and pass to the limit in the ``classical'' weak formulation of the equations,
where one drops the condition $\div \psi=0$ on the test-fuction.

\medbreak
The proof of the existence part of Theorem \ref{th:global-wp} is then completed.

\section{Proof of uniqueness} \label{s:uniqueness}

In this section, we carry out the proof of the uniqueness statement in Theorem \ref{th:global-wp}. In particular, performing this step will complete the proof of
the whole theorem.

Throughout this section, we assume the initial datum $\big(u_0,m_0\big)$ to be given and fixed, and to satisfy the assumptions of Theorem \ref{th:global-wp}.
We assume to have two solutions $\big(u_1,m_1\big)$ and $\big(u_2,m_2\big)$ associated to it and enjoying the properties listed in the same statement.
We stress the fact that the initial datum is the \emph{same} for both solutions. We will adopt the notation
\[
 \de u\,:=\,u_1\,-\,u_2\qquad\qquad \mbox{ and }\qquad\qquad \de m\,:=\,m_1\,-\,m_2\,.
\]
Notice that, if we denote (for $j=1,2$) by $\o_j\,:=\,\curl(u_j)$ the vorticity of the velocity field $u_j$, one has that
\[
\de\o\,:=\,\o_1\,-\,\o_2\,\equiv\,\curl(\de u)\,.
\]

Our proof of uniqueness relies on the same strategy used by Yudovich in his classical work \cite{Yud_1963}; we refer to \tsl{e.g.} Chapter 8 of \cite{Maj-Bert}
for details.
More precisely, we will perform an energy estimate (that is, a $L^2$ estimate) for the difference of the velocity fields $\de u$ and of the microrotation fields
$\de m$. The novelty here is that we have now to control possible growth due to the forcing terms appearing on the right-hand side of \eqref{eq:2D-microp};
however, these terms being linear, this will not be especially involved.

Notice that it does \emph{not} follow from our assumptions on the initial datum that the velocity field $u_j$, for $j=1,2$, belongs to $L^2$, see also Proposition
\ref{p:BS} in this respect. Thus, in order to carry out the plan described above, we need first to rigorously justify that we can indeed perform
an energy estimate on the difference $\de u$: this is the goal of the next statement.

\begin{lemma} \label{l:energy}
Let $p_0\in\,]1,2[\,$ and take two sets $\big(u_{0,j}, m_{0,j}\big)_{j=1,2}$ of initial data such that, for any $j\in\big\{1,2\big\}$, one has 
$\div u_{0,j}\,=\,0$ and, after defining $\o_{0,j}\,:=\,\curl\big(u_{0,j}\big)$,
\[
 \o_{0,j}\,,\,m_{0,j}\;\in\;L^{p_0}(\R^2)\,\cap\,L^\infty(\R^2)\,.
\]
Assume moreover that 
\[
 \de u_0\,:=\,u_{0,1}\,-\,u_{0,2}\;\in\,L^2(\R^2)\qquad\quad \mbox{ and }\qquad\quad \de m_0\,:=\,m_{0,1}\,-\,m_{0,2}\;\in\,H^1(\R^2)\,.
\]
Let $\big(u_j,m_j\big)_{j=1,2}$ be respective solutions related to those sets of initial data defined on a common time interval $[0,T]$,
for some $T>0$, and assume that they satisfy
\begin{align*}
& \o_j:=\curl(u_j)\,,\,m_j\;\in\;L^\infty\big([0,T];L^{p_0}(\R^2)\,\cap\,L^\infty(\R^2)\big)\,, \\
&\qquad\qquad\qquad \nabla m_j\,\in\,L^2\big([0,T];L^2(\R^2)\big)\,,
 \qquad u_j\,\in\,L^\infty\big([0,T];L^{q_0}(\R^2)\,\cap\,LL(\R^2)\big)\,,
\end{align*}
where the index $q_0$ is defined as $1/q_0\,=\,1/p_0\,-\,1/2$.

Then, the difference $\de u\,:=\,u_1-u_2$ belongs to the space $W^{1,2}\big([0,T];L^2(\R^2)\big)$, while the difference $\de m\,:=\,m_1-m_2$ belongs
to $\mc C\big([0,T];L^2(\R^2)\big)$. In addition,
for $f\,\in\,\big\{\de u, \de m\big\}$ and for any $t\in[0,T]$, one has the equality
\[
 \frac{\dd}{\dt}\left\|f(t)\right\|_{L^2}^2\,=\,2\,\int_{\R^2}\d_tf(t,x)\cdot f(t,x)\,\dx\,.
\]

\end{lemma}

\begin{proof}
We already know from the assumptions that $m_{0,j}$ belongs to $L^2$, for both $j=1$ and $j=2$, together with the properties
\[
\o_j\,,\, m_j\;\in\;L^\infty\big([0,T];L^2(\R^2)\big)\qquad\quad \mbox{ and }\qquad\quad \nabla m_j\,\in\,L^2\big([0,T];L^2(\R^2)\big)\,.
\]
At the same time, we also know that $u_j$ belongs to $L^\infty\big([0,T]\times\R^2\big)$.
Thus, an inspection of the equation for $m_j$ reveals that
\[
 \d_tm_j\,=\,-\,u_j\cdot\nabla m_j\,+\,\k\,\Delta m_j\,+\,\alpha\,\o_j
\]
belongs to the space $L^2\big([0,T];H^{-1}(\R^2)\big)$. 
From this, standard arguments (see \tsl{e.g.} Chapter 5 of \cite{Evans})
yield the property $m_j\,\in\,\mc C\big([0,T];L^2(\R^2)\big)$, together with the equality
\[
 \frac{\dd}{\dt}\left\|m_j(t)\right\|_{L^2}^2\,=\,2\,\int_{\R^2}\d_tm_j(t,x)\, m_j(t,x)\,\dx\,.
\]
Of course, the same properties are inherited by the difference $\de m$.

Let us now focus on the difference $\de u$ of the velocity fields, which is the most delicate quantity to be considered. First of all,
we remark that $\de u$ solves the transport problem
\begin{equation} \label{eq:de-u}
\d_t\de u\,+\,(u_1\cdot\nabla)\de u\,=\,g\,,\qquad\qquad \mbox{ with }\qquad \de u_{|t=0}\,=\,\de u_0\,,
\end{equation}
where we have defined
\[
 g\,:=\,-\,(\de u\cdot\nabla)u_2\,-\,\nabla\de\Pi\,-\,\alpha\,\nabla^\perp\de m\,,
\]
with $\nabla\de\Pi\,:=\,\nabla\Pi_1\,-\,\nabla\Pi_2$ being the difference of the pressure terms appearing in the equations of $u_1$ and $u_2$.
Observe that, by assumption, $g$ is a known forcing term, as all the differences appearing in its definition can be expressed in terms of the original
solutions $u_j$, $m_j$ and $\nabla\Pi_j$.

We now claim that
\begin{equation} \label{eq:bound-g}
\forall\,p\,\in\,[2,+\infty[\,,\qquad\qquad
 g\,\in\,L^2\big([0,T];L^p(\R^2)\big)\,.
\end{equation}
Indeed, arguing similarly to the proof of Lemma \ref{l:pressure}, it is easy to see that the sum $(\de u\cdot\nabla)u_2\,-\,\nabla\de\Pi$ belongs to those spaces
(as $p_0<2$).
Next, since each $\nabla m_j$ belongs to $L^2_T(L^2)$, we immediately have $\nabla^\perp \de m\,\in\,L^2_T(L^2)$ as well. In order to go above the integrability index
$p=2$, we consider the equation for $\de m$, which we write under the form of a heat equation with forcing term:
\begin{equation} \label{eq:de-m}
 \d_t\de m\,-\,\k\,\Delta\de m\,=\,\alpha\,\de\o\,-\,u_1\cdot\nabla\de m\,-\,\de u\cdot\nabla m_2\,,\qquad\qquad \mbox{ with }\qquad
 \de m_{|t=0}\,=\,\de m_0\,.
\end{equation}
Using that each $\o_j$ belongs to $L^\infty_T(L^{p_0}\cap L^\infty)$ together with the property $u_j\,\in\,L^\infty_T(L^\infty)$ and $\nabla m_j\,\in\,L^2_T(L^2)$ again,
we deduce that the right-hand side of the above equation belongs to the space $L^2_T(L^2)$. At the same time, the initial datum
$\de m_0$ belongs to $H^1$ by assumption. Since $H^1\,\hookrightarrow\,\dot H^1\,=\,\dot B^1_{2,2}$, we can apply maximal regularity estimates
(see \tsl{e.g.} \cite{LR}, see also Section 2B of \cite{D-F-P}) and infer that
\[
 \d_t\de m\,,\,\Delta\de m\;\in\;L^2\big([0,T];L^2(\R^2)\big)\,.
\]
By Sobolev embeddings in space dimension $d=2$, we then get that $\nabla\de m$ belongs to $L^2_T(L^p)$, for any $p\in[2,+\infty[\,$,
which in turn implies the claim \eqref{eq:bound-g}.

With property \eqref{eq:bound-g} at hand, we can solve the linear transport equation \eqref{eq:de-u}, with given initial condition $\de u_0$,
given transport field $u_1$ and given external force $g$. By assumption and Proposition \ref{p:BS}, we know that $u_{0,j}\in L^{q_0}$ for each $j$,
where $q_0\in\,]2,+\infty[\,$ is defined in that proposition. Since $q_0>2$, we also know that $g\in L^2_T(L^{q_0})$.
Therefore, we can solve (for instance, by smoothing out the transport field $u_1$ with respect to the space variable and then applying a compactness argument)
the linear transport problem \eqref{eq:de-u} in $L^2\cap L^{q_0}$ and get a solution $w\in L^\infty\big([0,T];L^2(\R^2)\,\cap\,L^{q_0}(\R^2)\big)$.

Observe that, since $u_1$ is divergence-free and $u_1\in L^\infty_T(H^1_\loc\,\cap\,L^\infty)$, we can apply DiPerna-Lions theory (see
in particular Theorem II.2 of \cite{DiP-L}) to deduce uniqueness of $L^p$ solutions to the linear transport equation \eqref{eq:de-u}.
Now, because $\de u$ is another solution of the same transport problem \eqref{eq:de-u} belonging to the same space $L^\infty\big([0,T];L^{q_0}(\R^2)\big)$, we must have
\[
 \de u\,\equiv\,w\qquad\qquad \mbox{ a.e. on }\quad [0,T]\times\R^2\,.
\]
In turn, since we also have $w\in L^\infty_T(L^2)$, from the above a.e. equality we deduce that
\[
 \de u\,\in\,L^\infty\big([0,T];L^2(\R^2)\big)\,.
\]
Next, since $(u_1\cdot\nabla)\de u$ belongs to $L^\infty_T(L^2)$ by the same token as the one used for the first terms appearing in the definition of $g$,
an inspection of equation \eqref{eq:de-u} yields that
\[
 \d_t\de u\,\in\,L^2\big([0,T];L^2(\R^2)\big)\,,\qquad\qquad \mbox{ thus }\qquad \de u\,\in\,W^{1,2}\big([0,T];L^2(\R^2)\big)\,.
\]

Recalling that $\de\o\in L^\infty_T(L^2)$, we deduce $\de u\,\in\,L^\infty\big([0,T];H^1(\R^2)\big)\,\cap\,W^{1,2}\big([0,T];L^2(\R^2)\big)$.
Then, we can apply again the arguments of Chapter 5 of \cite{Evans} to deduce the equality
\[
 \frac{\dd}{\dt}\left\|\de u(t)\right\|_{L^2}^2\,=\,2\,\int_{\R^2}\d_t\de u(t,x)\cdot \de u(t,x)\,\dx\,.
\]
The proof of the lemma is now complete.
\end{proof}

Lemma \ref{l:energy} plays a crucial role in the proof of the uniqueness statement in Theorem \ref{th:global-wp}, which we can now tackle.
As already mentioned, this relies on an energy estimates for the quantities $\de u$ and $\de m$, combined with the classical
Yudovich uniqueness argument \cite{Yud_1963, A-B-C-DL-G-J-K} (see also \cite{F_2025_Yud} for similar computations).

\begin{proof}[Proof of uniqueness in Theorem \ref{th:global-wp}]
Let $\big(u_0,m_0\big)$ be a given initial datum satisfying the assumptions of Theorem \ref{th:global-wp}.
By contradiction, we assume to have two solutions $\big(u_1,m_1\big)$ and $\big(u_2,m_2\big)$ associated to it and enjoying the properties listed in the same statement.

We stress the fact that the initial datum is the \emph{same} for both solutions.
Therefore, adopting the same notation of Lemma \ref{l:energy} above, one has that
\[
 \de u_0\,\equiv\,0\qquad \mbox{ and }\qquad \de m_0\,\equiv\,0\qquad\qquad \mbox{ a.e. on }\quad \R^2\,.
\]
The assumptions of Lemma \ref{l:energy} are then satisfied. Its conclusions imply in particular that we can perform an energy estimate
in the equations \eqref{eq:de-u} for $\de u$ and \eqref{eq:de-m} for $\de m$. We observe here that all the integration by parts with respect
to the space variable can be rigorously justified by similar arguments as in Subsection 5.2 of \cite{F_2025_Yud}.

For simplicity of notation, here below we are going to set $\alpha=\k=1$. All the integrals appearing below will be space integrals,
therefore we will omit to denote the space domain $\R^2$ in the notation. Finally, for any $t\geq0$, we denote
\[
E(t)\,:=\,\left\|\de u(t)\right\|_{L^2}^2\,+\,\left\|\de m(t)\right\|_{L^2}^2\,.
\]

We start by multiplying equation \eqref{eq:de-u} by $\de u$ to get, for all $t\geq0$, the equality
\[
 \frac{1}{2}\,\frac{\dd}{\dt}\left\|\de u(t)\right\|_{L^2}^2\,=\,-\int(\de u\cdot\nabla)u_2\cdot\de u\,\dx\,-\,\int\nabla^\perp \de m\cdot \de u\,\dx\,,
\]
where we have also used the $L^2$-orthogonality between $\de u$ (which is divergence-free) and $\nabla\de\Pi$.
From the equation \eqref{eq:de-m} multiplied by $\de m$, instead, we find
\begin{align*}
 \frac{1}{2}\,\frac{\dd}{\dt}\left\|\de m(t)\right\|_{L^2}^2\,+\,\left\|\nabla \de m\right\|_{L^2}^2\,=\,\int\de\o\,\de m\,\dx\,-\,\int\de u\cdot\nabla m_2\,\de m\,\dx\,.
\end{align*}
Observe that an integration by parts (which is well justified here, see \cite{F_2025_Yud} again for details) gives
\[
\int\de\o\,\de m\,\dx\,=\,\int \de u^\perp\cdot\,\nabla \de m\,\dx\,=\,-\,\int\de u\cdot\nabla^\perp\de m\,\dx\,.
\]
Thus, gathering the above relations together, we obtain
\begin{align*}
 \frac{1}{2}\,\frac{\dd}{\dt}E(t)\,+\,\left\|\nabla \de m\right\|_{L^2}^2\,&\leq\,\left|\int(\de u\cdot\nabla)u_2\cdot\de u\,\dx\right| \\
 &\qquad\qquad \,+\,
2\,\left|\int\de u\cdot\nabla^\perp\de m\,\dx\right|\,+\,\left|\int\de u\cdot\nabla m_2\,\de m\,\dx\right|\,.
\end{align*}

Now, we focus on the second term on the right-hand side of the previous relation, for which we use the Cauchy-Schwarz inequality together with the Young inequality.
Doing so, we can absorbe the $\nabla\de m$ appearing in second term on the right-hand side
into the left-hand side and infer that
\begin{align*}
\frac{\dd}{\dt}E(t)\,+\,\left\|\nabla \de m\right\|_{L^2}^2\,&\lesssim\,E(t)\,+\,\left|\int(\de u\cdot\nabla)u_2\cdot\de u\,\dx\right|\,+\,
\left|\int\de u\cdot\nabla m_2\,\de m\,\dx\right|\,,
\end{align*}
for an implicit multiplicative constant $C>0$ only depending on $\alpha$ and $\kappa$ (set both equal to $1$ in these computations).
Next, we focus on the last term appearing on the right-hand side of the previous relation. Since $\de u$ is divergence-free and we dispose of enough regularity
to integrate by parts, we can write
\begin{align*}
\left|\int\de u\cdot\nabla m_2\,\de m\,\dx\right|\,&=\,\left|\int\de u\cdot\nabla \de m\, m_2\,\dx\right| \\
&\leq\,\left\|m_2\right\|_{L^\infty}\,\left\|\de u\right\|_{L^2}\,\left\|\nabla \de m\right\|_{L^2}\,.
\end{align*}
Using estimate \eqref{est:G-m_final} for $m_2$, we can bound
\[
 \left\|m_2(t)\right\|_{L^\infty}\,\lesssim\,\left\|\big(\o_0\,,\,m_0\big)\right\|_{L^{p_0}\cap L^\infty}\,e^{C\,t}\,.
\]
Another application of the Young inequality then yields the estimate
\begin{align} \label{est:de-E}
\frac{\dd}{\dt}E(t)\,+\,\left\|\nabla \de m\right\|_{L^2}^2\,&\leq\,C_1\,e^{C\,t}\,E(t)\,+\,C_2\,\left|\int(\de u\cdot\nabla)u_2\cdot\de u\,\dx\right|\,,
\end{align}
for suitable universal multiplicative constants $C$ and $C_2$, not depending on the data nor on the solution, nor on the time $t$, and
another constant $C_1$, depending only on the $L^{p_0}\cap L^\infty$-norm of the initial vorticity $\o_0$ and microrotation field $m_0$.

The bound of the last term on the right-hand side of \eqref{est:de-E} is pretty standard and can be performed as in the classical Yudovich proof. For
completeness, let us sketch the argument. 
Let us recall that, owing to the fact that $\o_2\in L^\infty_T(L^{p_0}\cap L^\infty)$ and to the Calder\'on-Zygmund theory, we have the property
$\nabla u_2\in L^\infty_T(L^p)$ for all $p_0\leq p<+\infty$.
On the other hand, one has $\de u \in L^\infty_T(L^2)$, because of Lemma \ref{l:energy}, and also $\de u\in L^\infty_T(L^\infty)$, as each $u_j$
belongs to this space. Therefore, taken $p>p_0$ large and $q$ such that $1/p\, +\, 1/q\, =\, 1/2$, we can bound
\begin{align*}
\left|\int_{\R^2}(\de u\cdot\nabla) u_2\cdot\de u\,\dx\right|\,\leq\,\left\|\nabla u_2\right\|_{L^p}\,\left\|\de u\right\|_{L^q}\,
\left\|\de u\right\|_{L^2}\,.
\end{align*}
We now use an interpolation argument to control the $L^q$ norm of $\de u$: 
after setting $\theta \,=\,2/q\,=\,1\,-\,2/p$, we can write
\begin{align*}
 \left\|\de u(t)\right\|_{L^q}\,&\leq\,\left\|\de u\right\|_{L^2}^\theta\;\left\|\de u\right\|^{1-\theta}_{L^\infty}\,.
\end{align*}
For the term depending on the $L^\infty$ norm of $\de u$, we can use estimate \eqref{est:u_final}. Observe that this term depends on
$p$ through the exponent $1-\theta$; however, it can be easily bounded uniformly with respect to $p\geq p_0$ by an application
of the Young inequality. Therefore, we get
\[
 \left\|\de u(t)\right\|_{L^q}\,\lesssim\;e^{C\,t}\,\left\|\de u\right\|_{L^2}^\theta\,,
\]
where, now, the implicit multiplicative constant depends also on the $L^{p_0}\cap L^\infty$ norms of $\o_0$ and $m_0$.
Since, owing to the Calder\'on-Zygmund theory, we have
\[
\forall\,p\in\,]1,+\infty[\,,\qquad\qquad \left\|\nabla u\right\|_{L^p}\,\lesssim\,\frac{p^2}{p-1}\,\left\|\o\right\|_{L^p}\,,
\]
in turn we find, thanks to \eqref{est:o_final}, the inequality
\begin{align} \label{est:CZ-term}
\left|\int_{\R^2}(\de u\cdot\nabla) u_2\cdot\de u\,\dx\right|\,\leq\,C_3\,\frac{p^2}{p-1}\,e^{2\,C\,t}\,\left\|\de u\right\|_{L^2}^{1+\theta}\,,
\end{align}
where $C_3$ depends on the Lebesgue norms of the initial data $\o_0$ and $m_0$ and where $\theta\,=\,2/q\,=\,1\,-\,2/p$ as above.
Plugging \eqref{est:CZ-term} into \eqref{est:de-E} then yields, for any $t\geq0$ and for any value of $p\geq p_0$, the bound
\begin{align} \label{est:de-E_p}
 \frac{\dd}{\dt}E(t)\,\leq\,C_0\,e^{2\,C\,t}\,\left(E(t)\,+\,p\,\Big(E(t)\Big)^{1-\frac1p}\right)\,,
\end{align}
for two suitable constants $C_0>0$ and $C>0$, with the former depending on the $L^{p_0}\cap L^\infty$ norm of $\o_0$ and $m_0$ and the latter being
universal (independent of the data and solutions, as well as of the time $t\geq 0$).

In order to treat the presence of the linear term $E(t)$ in the above inequality \eqref{est:de-E_p},
we proceed similarly to Subsection 5.2 of \cite{F_2025_Yud}. To begin with, we define the time $T_0>0$ as
\[
 T_0\,:=\,\sup\left\{t\geq 0\;\Big|\quad E(t)\,\leq\,\frac{1}{2}\right\}\,.
\]
Observe that $\de u$ and $\de m$ are continuous in time with values in $L^2$ (see Lemma \ref{l:energy}),
hence $E$ is a continuous function of time. Since $E(0)=0$ by assumption, $T_0$ is well-defined and one has indeed $T_0>0$.
Without loss of generality, we can assume that $T_0<+\infty$ (if $T_0=+\infty$, simply fix any finite time $T_1>0$).
In light of \eqref{est:de-E_p}, since $p\geq p_0>1$, on the time interval $[0,T_0]$ we get
\begin{equation} \label{est:dE-dt}
 \frac{\dd}{\dt}E(t)\,\leq\,K\,p\,\Big(E(t)\Big)^{1-\frac1p}\,,
\end{equation}
where we have set $K\,:=\,2\,C_0\,e^{2\,C\,T_0}$. Despite the previous inequality has not a unique solution for $p$ fixed, it does possess a maximal solution
\[
 E_{\max}(t)\,=\,\Big(K\,t\Big)^p\,,
\]
so that any other solution of \eqref{est:dE-dt} satisfies $E(t)\,\leq\,E_{\max}(t)$ for all $t\geq0$. We refer to Lemma B.2 of \cite{A-B-C-DL-G-J-K}
for more details. At this point, we define the time $T_*>0$ as
\[
 T_*\,:=\,\min\left\{T_0\,,\,\frac{1}{2K}\right\}\,.
\]
Then, from \eqref{est:dE-dt} and the previous considerations we deduce the bound 
\[
\sup_{t\in[0,T_*]} E(t)\,\leq\,\Big(K\,T_*\Big)^p\,\leq\,2^{-p}\,.
\]
Remark that the previous bound holds true for \emph{any} value of $p\geq p_0$. Then, sending $p\to+\infty$ yields that
$E\equiv0$ on $[0,T_*]$. From this, it is then a routine matter to deduce that
$E\equiv 0$ on the whole time interval $[0,T_0]$ and then on the whole $\R_+$.

This completes the proof of the uniqueness statement in Theorem \ref{th:global-wp}, hence also the proof of the whole theorem.
\end{proof}


%


\end{document}